\documentclass[12pt, reqno, a4paper]{amsart}

\usepackage{ amssymb, amsmath, enumerate, amsfonts, amsthm, mathrsfs, url, bm, mathtools}

\usepackage{xcolor}  	
\usepackage{hyperref}
\hypersetup{
colorlinks,
   linkcolor={cyan!80!black},
   citecolor={cyan!80!black},
 urlcolor={cyan!80!black}
}

\usepackage{color}

\usepackage[margin=1in]{geometry}

\RequirePackage{doi}

\usepackage{amscd}
\usepackage{amsfonts}
\usepackage{float}
\usepackage{color}
\usepackage[
backend=biber,
style=alphabetic,
]{biblatex}
\usepackage{bookmark}

\renewbibmacro{in:}{}
\DeclareFieldFormat{title}{#1}

\DeclareFieldFormat[article]{title}{\mkbibemph{#1}}       
\DeclareFieldFormat[incollection]{title}{\mkbibemph{#1}}  
\DeclareFieldFormat[book]{title}{\mkbibemph{#1}}          
\DeclareFieldFormat[incollection]{booktitle}{#1}          
\DeclareFieldFormat[article]{journaltitle}{#1}            

\AtEveryBibitem{%
  \ifentrytype{misc}{\DeclareFieldFormat{title}{\mkbibemph{#1}}}{}}

\DeclareFieldFormat{eprint:eprint}{arXiv:\href{https://arxiv.org/abs/#1}{#1}}

\DeclareFieldFormat[inproceedings]{title}{\mkbibemph{#1}}

\DeclareFieldFormat[inproceedings]{booktitle}{#1}

\usepackage{amssymb}

\newtheorem{theorem}{Theorem}[section]
\newtheorem{lemma}{Lemma}[section]

\newtheorem{corollary}{Corollary}[section]
\newtheorem{proposition}{Proposition}[section]

\theoremstyle{definition}
\newtheorem{definition}{Definition}[section]

\newtheorem{conjecture}{Conjecture}[section]

\theoremstyle{remark}
\newtheorem{remark}{Remark}[section]

\numberwithin{equation}{section}

\newcommand{\Mod}[1]{\ (\mathrm{mod}\ #1)}

\renewcommand{\Re}{\mathrm{Re}}

\newcommand{\Cov}{\mathrm{Cov}}

\newcommand{\meas}{\mathrm{meas}}
\newcommand{\sgn}{\mathrm{sgn}}
\newcommand{\Li}{\mathrm{Li}}

\renewcommand{\leq}{\leqslant}
\renewcommand{\geq}{\geqslant}

\begin{document}

\title[Joint distribution of primes in multiple short intervals]{Joint distribution of primes in multiple short intervals}


\author{}
\address{}
\curraddr{}
\email{}
\thanks{}

\author{Sun-Kai Leung}
\address{D\'epartement de math\'ematiques et de statistique\\
Universit\'e de Montr\'eal\\
CP 6128 succ. Centre-Ville\\
Montr\'eal, QC H3C 3J7\\
Canada}
\curraddr{}
\email{sun.kai.leung@umontreal.ca}
\thanks{}


\date{}

\dedicatory{}


\begin{abstract}
Assuming the Riemann hypothesis (RH) and the linear independence conjecture (LI), we show that the weighted count of primes in multiple short intervals follows a multivariate Gaussian distribution with weak negative correlations. As an application, we obtain short-interval analogues of many results in the literature on the Shanks–Rényi prime number race, including a sharp phase transition: biased races between primes in short intervals emerge once the number of intervals exceeds an explicit critical threshold. Our result is new even for a single moving interval, particularly under a quantitative formulation of the linear independence conjecture (QLI).
\end{abstract}

\maketitle

\setcounter{tocdepth}{1}
\tableofcontents

\section{Introduction}

The study of primes in short intervals can be traced back to  the late 18th century when the young prodigy Gauss examined tables of primes in search of patterns. In 1792, at the age of 15, he made a guess that despite its fluctuations,  the density of primes near 
$x$ is approximately $1/\log x$, as observed from counting primes in intervals of length $1000$ (chiliads). However, this prediction took more than a century to  rigorously justify.
Eventually, in 1896, Hadamard and de la Vall\'{e}e Poussin independently resolved this conjecture, which is now known as the \textit{Prime Number Theorem}. It states that as $x \to \infty,$ we have
\begin{align*}
\pi(x):=\# \{ p \leq x\} \sim \Li(x):=\int_2^{x} \frac{dt}{\log t},
\end{align*}
or equivalently, the weighted count of primes (and prime powers) satisfies
\begin{align*}
\psi(x):=\sum_{p^k \leq x} \log p \sim x.
\end{align*}
Consequently, one can show that the mean
of $\psi(n+H)-\psi(n)$ as $n \in [1,N]$ varies is
\begin{align*}
\frac{1}{N}\sum_{n \leq N} \left(\psi(n+H)-\psi(n)\right) 
\sim H,
\end{align*}
provided that $ H =o (N)$ as $N \to \infty.$


In his famous letter to Encke, 
Gauss counted primes in intervals of length $100$ from 1 million to 3 million, and also
compared $\pi(n+H)-\pi(n)$ and $\Li(n+H)-\Li(n)$ with $H=10^5$ but did not pursue it beyond comparisons with Legendre’s estimates. Given the wild fluctuations of prime counts in short intervals, what is their statistical behavior? 

In 1973, Goldston and Montgomery \cite{MR1018376} showed that the variance of $\psi(n+H)-\psi(n)$ as $n \in [1, N]$ varies is $\sim H\log \frac{N}{H}$ in the range of $H \in \left[ N^{\varepsilon}, N^{1-\varepsilon} \right] $ under the \textit{Riemann hypothesis} (RH) and the \textit{strong pair correlation conjecture} \cite{MR0337821}, which respectively concern the horizontal and vertical distribution of zeros of the Riemann zeta function.

Relatively recently, in 2004, Montgomery and Soundararajan \cite{MR2104891} showed that the distribution of $\psi(n+H)-\psi(n)$ as $n \in [1, N]$ varies is asymptotically normal with mean 
$\sim H$ and variance $\sim H\log \frac{N}{H},$ provided that
$\frac{H}{\log N} \to \infty$ and $\frac{\log H}{\log N} \to 0$ as $N \to \infty,$ by computing higher moments
under a uniform \textit{Hardy–Littlewood prime $k$-tuple conjecture} \cite{MR1555183}, which concerns the correlation of primes. They further conjectured that in the range of
$H \in [(\log N)^{1+\delta}, N^{1-\delta}],$ the normality persists.

In this paper, we revisit the \textit{Fourier side}.\footnote{That is, we apply Fourier analysis.} Instead of the pair correlation conjecture---an analytic assumption---we adapt the method of Rubinstein and Sarnak~\cite{MR1329368}, assuming the linear independence over~$\mathbb{Q}$ of the positive ordinates of the nontrivial zeros (LI), an algebraic condition. To our knowledge, this is the first instance where the emphasis is not on the distribution of primes or irreducible polynomials in residue classes, but rather on their distribution in short intervals (see \cite{Martin_2025} for an annotated bibliography).

As we shall see, given a large $X$ and $x \in [2,X],$ 
suppose $h=h(x)=\delta x,$ where $\delta>0$ is small but independent of $X,$ which is beyond the conjectural range of Montgomery and Soundararajan stated above. Then the distribution of $\psi(x+h)-\psi(x)$ as $x \in [2, X]$ varies (in logarithmic scale) remains Gaussian under RH and LI (see 
Theorem \ref{thm:clt} and 
Theorem \ref{cor:compare} with $r=1$).

Furthermore, one may ask: What is the joint distribution of the weighted count of primes in two neighboring intervals? Are they independent? If not, how are they correlated? 
We show that, assuming RH and LI, the pair
$\left( \psi(x)-\psi(x-h), \psi(x+h)-\psi(x) \right)$
as $x \in [2, X]$ varies (in logarithmic scale) is asymptotically bivariate Gaussian with a weak negative correlation. 
\begin{corollary} \label{thm:neighbour}
Assume RH and LI. Given a Borel subset $B \subseteq \mathbb{R}^2,$ define
\begin{align*}
 S_{X,\delta;B}:=\left\{x \in [2,X] \,: \,
\frac{\left( \psi(x)-\psi(x-\delta x)-\delta x, 
\psi(x+\delta x)-\psi(x)-\delta x\right)}
{\sqrt{\left(\delta \log \frac{1}{\delta} + 
(1-\gamma_{\mathbb{Q}}-\log 2\pi)\delta\right)x}}
\in B \right\}.
\end{align*}
Then for any $\delta>0$ sufficiently small, we have
\begin{align*}
\lim_{X \to \infty}\frac{1}{\log X}\int_{S_{X,\delta;B}} \frac{dx}{x}
=&
\frac{1}{2\pi\sqrt{\det\mathcal{C}}} \int_{B}
\exp \left( -\frac{1}{2}\langle \mathcal{C}^{-1}\boldsymbol{x},\boldsymbol{x}\rangle  \right) d\boldsymbol{x} +O\left( \frac{1}{\log^2 \frac{1}{\delta}} \right)
\end{align*}
with the covariance matrix 
\begin{align*}
\mathcal{C}=
\begin{pmatrix}
1 & -\frac{\log 2}{\log \frac{1}{\delta}}  \\
-\frac{\log 2}{\log \frac{1}{\delta}} & 1
\end{pmatrix}
.
\end{align*}
\end{corollary}

In particular, primes in two neighboring intervals are aware of and avoid each other. More generally, we show that the weighted count of primes in multiple disjoint short intervals has an asymptotically multivariate Gaussian distribution (in logarithmic scale), with weak negative correlations, under RH and LI (see Theorem \ref{thm:clt} or Theorem \ref{cor:compare}).

\begin{remark}
To derive Corollary \ref{thm:neighbour} from Theorem \ref{thm:clt}, one can simply adapt the proof of Corollary \ref{cor:negcorr} (see Section \ref{sec:pfofcor} for more details).
\end{remark}

In Section \ref{newsection}, further assuming a quantitative 
formulation of the LI hypothesis, we demonstrate that for primes in a single moving interval, the normality persists as long as $\delta > (\log X)^{-\varepsilon}$ for any $\varepsilon>0$ (see Theorem \ref{thm:short}).

\section*{Notation and convention}

Throughout the paper, we adopt the following notations and conventions:

\begin{itemize}
\item we say $f(x) = O(g(x))$ or $f(x) \ll g(x) $ if there exists a constant $C > 0$ depending on the subscripted parameters such that $|f(x)| \leq Cg(x)$ for all $x$ in the domain of $f$;

\item we say $f(x)=o(g(x))$ as $x \to \infty$ or $x \to 0^+$ if $\frac{f(x)}{g(x)} \to 0$ in the corresponding limit, where the rate of convergence depends on the subscripted parameters;

\item we say $f \asymp g$ if there exist constants $C_1,C_2 > 0$ depending on the subscripted parameters
such that $C_1 g(x) \leq |f(x)| \leq C_2 g(x)$ for all $x$ in the domain of $f;$


\item
$[r]:=\{1,\ldots,r\};$

\item $\|\boldsymbol{x}\|=\|\boldsymbol{x}\|_2:=\sqrt{x_1^2+\cdots+x_r^2};$

\item $\|\boldsymbol{x}\|_{\infty}:=\max\{|x_1|,\ldots,|x_r|\};$

\item $\meas(B)$ denotes the Lebesgue measure of a Borel subset $B \subseteq \mathbb{R}^r;$

\item $\gamma_{\mathbb{Q}}$ denotes the Euler--Mascheroni constant;

\item $\Lambda(n):=
\begin{cases}
\log p & \mbox{{\normalfont if $n=p^k$ for some integer $k \geq 1,$} }\\
\hfil 0 & \mbox{{\normalfont otherwise;} }
\end{cases}
$

\item $\psi(x):=\sum_{n \leq x}\Lambda(n);$

\item $\rho=\frac{1}{2}+i\gamma$ denotes the nontrivial zero of the Riemann zeta function at height $\gamma$;

\item $N(T):=\#\{0 <\gamma \leq T \};$ 

\item $\mathcal{N}(\boldsymbol{0}, \mathcal{C})$ denotes a multivariate Gaussian with mean $\boldsymbol{0}$ and covariance matrix $\mathcal{C}.$




\end{itemize}

\section*{Symbol index}
The following index lists defined symbols and the pages where they first appear.
\begin{itemize}
    \item $E(x;\delta,t), T_S, w(s;\delta,t), \mathbb{P}_x^{\log}( \boldsymbol{Y}(x) \in B ), \mu_{\delta,\boldsymbol{t}}$ \dotfill p.\pageref{not:Exdeltat}
  \item $X_{\delta,t}, \Cov_{jk}(\delta,\boldsymbol{t}), V_j, \Delta(t)$ \dotfill p.\pageref{not:xdt}
  \item $\widetilde{\boldsymbol{E}}(x;\delta,\boldsymbol{t}), c_{jk}(\delta,\boldsymbol{t})$ \dotfill p.\pageref{not:renormalizeddeviation}
  \item $R_{\boldsymbol{\alpha}, \boldsymbol{\beta}}$ \dotfill p.\pageref{not:Rab}
  \item $\rho(\delta;\boldsymbol{t})$ \dotfill p.\pageref{not:rhodeltaboldsymbolt}
  \item $\rho_s(\delta,\boldsymbol{t})$ \dotfill p.\pageref{not:rhosdeltaboldsymbolt}
\end{itemize}

\section{Preliminaries}

Let us introduce several quantities that will appear throughout the paper.

Given an integer $r \geq 1,$ real numbers $x \geq 2, \delta>0$ and a vector $\boldsymbol{t} \in \mathbb{R}^r,$ we denote by
$\boldsymbol{E}(x;\delta,\boldsymbol{t})$ the $r$-tuple
$(E(x;\delta,t_1), \ldots, E(x;\delta,t_r) ),$ where
\begin{align*} 
E(x;\delta,t)
:=\frac{1}{\sqrt{x}}
\left( \psi\left( (1+t \delta)x+\frac{1}{2}\delta x \right) -
\psi\left( (1+t\delta)x-\frac{1}{2}\delta x \right) -\delta x
\right),
\end{align*}
\label{not:Exdeltat}i.e., the normalized deviation of the weighted prime count in the short interval of length $\delta x$ centered at $(1+t\delta )x.$ 
Here, we always assume $\delta>0$ is sufficiently small, so that
$2 \leq (1+t\delta)x-\delta x/2 \leq (1+t\delta)x+\delta x/2  \leq 2x .$ 
To simplify our discussion, we also require that $|t_j-t_k|\geq 1$ whenever $j\neq k,$ i.e., the intervals are disjoint.

\begin{remark}
In logarithmic scale, it is not unnatural to study the distribution of primes in ``multiplicative intervals" $[\exp((t-\frac{1}{2})\delta)x, \exp(t+\frac{1}{2})\delta)x]$ instead of ``additive intervals" $[(1+t\delta)x-\frac{1}{2}\delta x, (1+t\delta)x+\frac{1}{2}\delta x].$ In particular, calculations in the proof of Proposition \ref{thm:cov} can be simplified. Nevertheless, the difference is minimal, since $\exp(\delta) = 1 + \delta + O(\delta^2)$ for $\delta>0$ sufficiently small.
\end{remark}

Given a subset $S$ of $[r]:=\{1,\ldots,r\},$ we define\label{not:ts}
\begin{align*}
T_S:=1+\max_{j \in S} |t_j|.
\end{align*}


For $s \in \mathbb{C},$ we define
\begin{align} \label{eq:wsdt}
w(s)=w(s;\delta,t):=\frac{1}{s}\left( \left( 
1+\left( t+\frac{1}{2} \right)\delta
\right)^{s}-\left( 1+\left( t-\frac{1}{2} \right)\delta \right)^{s} \right)
\end{align}
\label{not:wsdeltat}and write $w_j(s):=w(s;\delta,t_j)$ for $j=1,\ldots,r.$


Inspired by the work of Rubinstein and Sarnak \cite{MR1329368} on primes in multiple arithmetic progressions to a large modulus, we study the joint distribution of primes in multiple short intervals with respect to the logarithmic density under RH and LI.

\begin{definition} \label{def}
Let $\boldsymbol{Y}(x)$ be a $\mathbb{R}^r$-valued function. We say that $\boldsymbol{Y}(x)$ has a logarithmic limiting distribution $\mu$ on $\mathbb{R}^r$ if 
\begin{align*}
\mathbb{E}_{x}^{\log}\left( f
\left( \boldsymbol{Y}(x) \right)\right):=&
\lim_{X \to \infty}\frac{1}{\log X}\int_{2}^X
f(\boldsymbol{Y}(x))\frac{dx}{x}\\
=&\lim_{U \to \infty}\frac{1}{U}\int_{1}^U
f(\boldsymbol{Y}( e^u))du\\
=&\int_{\mathbb{R}^r}f(\boldsymbol{y})d\mu(\boldsymbol{y})
\end{align*}
for all bounded continuous functions $f$ on $\mathbb{R}^r,$
i.e., the logarithmic time average equals the space average with respect to the measure $\mu$.

If such a measure $\mu$ exists, then for any Borel subset $B \subseteq \mathbb{R}^r,$ we define
\begin{align*}
\mu(B):=\mathbb{P}_x^{\log}\left( \boldsymbol{Y}(x) \in B \right)
=\mathbb{E}_{x}^{\log}\left( 1_B
\left( \boldsymbol{Y}(x) \right)\right),
\end{align*}
\label{not:mathbbPlog}where $1_B(\boldsymbol{x})$ is the indicator function of the Borel set $B.$
\end{definition}

In preparation for our main result Theorem \ref{thm:clt}, we establish the following propositions.

\begin{proposition}  \label{thm:rs}
Let $r \geq 1,\delta>0$ and $\boldsymbol{t} \in \mathbb{R}^r$ be fixed. Assume RH and LI. Then 
$\boldsymbol{E}(x;\delta,\boldsymbol{t})$ has a logarithmic limiting distribution $\mu_{\delta,\boldsymbol{t}}$\label{not:mudeltaboldsymbolt} on $\mathbb{R}^r.$ Furthermore, it is an absolutely continuous probability measure corresponding to the $\mathbb{R}^r$-valued random vector $\boldsymbol{X}_{\delta,\boldsymbol{t}}=(X_{\delta,t_1},\ldots,X_{\delta,t_r}),$ where\label{not:xdt}
\begin{align*}
X_{\delta,t}:=\Re \left(2 \sum_{\gamma>0} w(\rho)U_{\gamma}\right)
\end{align*}
with the sum running over the positive ordinates of the nontrivial zeros, and $\{U_{\gamma}\}_{\gamma>0}$ being a sequence of independent random variables uniformly distributed on the unit circle $\mathbb{T}.$
Moreover, the covariance matrix of $\boldsymbol{X}_{\delta,\boldsymbol{t}}$ is real symmetric with the $(j,k)$-entry being
\begin{align*}
\Cov_{jk}=\Cov_{jk}(\delta,\boldsymbol{t}):=
\sum_{\gamma}w_j(\rho)\overline{w_k(\rho)}.
\end{align*}
\end{proposition}

To lighten the notation, for $j=1,\ldots,r,$ we denote the variance $\Cov_{jj}$ by $V_j.$
\begin{proposition} \label{thm:cov} 
Assume RH and LI. Let $\delta>0$ be sufficiently small and $T=T_{\{j,k\}} \leq 
 \delta^{-\frac{1}{2}}\left(\log\frac{1}{\delta} \right)^{-1}.$ Then 
\begin{align*}
\Cov_{jk}=
\begin{cases}
\delta \log \frac{1}{\delta}
+(1-\gamma_{\mathbb{Q}}-\log 2\pi)\delta +O\bigl( \left(T \delta \log \frac{1}{ \delta}\right)^2 \bigr) & \mbox{{\normalfont if $j=k,$ } }\\
\hfil -\Delta(|t_j-t_k|)\delta
+O\bigl( \left(T \delta \log \frac{1}{ \delta}\right)^2 \bigr) & \mbox{{\normalfont if $j\neq k,$ } }
\end{cases}
\end{align*}
where 
\begin{align} \label{eq:biddelta}
\Delta(t):=\frac{1}{2}\biggl( (t+1)\log(t+1)
-2t\log t+(t-1)\log(t-1) \biggr),
\end{align}
i.e., the second order central difference of the function $f(t)=\frac{1}{2}t\log t.$
\end{proposition}

\begin{remark} In particular, if $\delta>0$ is sufficiently small, then $V_j=\Cov_{jj}\leq \delta \log \frac{1}{\delta}$ for $j=1,\ldots,r.$
\end{remark}

Note that \( \Delta(1) = \log 2 \) and \( \Delta(|t_j - t_k|) \) is always positive. Moreover, as \( |t_j - t_k| \to \infty \), we have \( \Delta(|t_j - t_k|) \to 0^+ \) monotonically. More precisely, we have (see Lemma \ref{lem:eq:coulomb})
\begin{align*} 
\Delta(|t_j - t_k|) =\frac{1+o(1)}{2|t_j - t_k|}.
\end{align*}
Thus, primes in disjoint short intervals obey a Coulomb-like law:\footnote{The Coulomb potential describes the electrostatic energy between two point charges as inversely proportional to their separation.} they repel each other, albeit weakly, with an intensity approximately inversely proportional to their separation.

\begin{remark}
Applying the elementary identity 
\begin{align*}
xz=\frac{1}{2}\left( (x+y+z)^2-(x+y)^2-(y+z)^2+y^2 \right)
\end{align*}
with (assuming $t_2>t_1$)
\begin{align*}
x=& \psi\left( (1+t_1 \delta)x+\frac{1}{2}\delta x \right) -
\psi\left( (1+t_1\delta)x-\frac{1}{2}\delta x \right) -\delta x,\\
y=& \psi\left( (1+t_2 \delta)x-\frac{1}{2}\delta x \right) -
\psi\left( (1+t_1\delta)x+\frac{1}{2}\delta x \right) -(t_2-t_1-1)\delta x, \\
z=&\psi\left( (1+t_2 \delta)x+\frac{1}{2}\delta x \right) -
\psi\left( (1+t_2\delta)x-\frac{1}{2}\delta x \right) -\delta x,
\end{align*}
all two-point correlations of weighted prime count in short intervals
can be determined via variances, providing a quick explanation for the repulsion. However, since higher mixed moments cannot be determined from pure moments, it is unclear a priori whether the weighted prime count in short intervals follows a multivariate normal distribution.
\end{remark}

\begin{remark}
The appearance of the secondary term \( (1 - \gamma_{\mathbb{Q}} - \log 2\pi)\delta \) is expected, as the variance computed by Montgomery and Soundararajan \cite{MR1954710} is
\[
\frac{1}{X} \int_{1}^X \left( \psi(x+H) - \psi(x) - H \right)^2 \, dx = H \log \frac{X}{H} - (\gamma_{\mathbb{Q}} + \log 2\pi+o(1))H 
\]
for \( X^{\varepsilon} \leq H \leq X^{1/2 - \varepsilon} \), under a uniform Hardy--Littlewood prime \( k \)-tuple conjecture. Chan \cite{MR2703577} also derived the same expression under a refinement of the strong pair correlation conjecture.
\end{remark}

\section{Main results}

In view of Proposition \ref{thm:rs}, we shall state our main theorems in terms of the renormalized deviation\label{not:renormalizeddeviation} 
\begin{align*}
\widetilde{\boldsymbol{E}}(x;\delta,\boldsymbol{t}):=
\left( \frac{E(x;\delta,t_1)}{\sqrt{V_1}},\ldots,
\frac{E(x;\delta,t_r)}{\sqrt{V_r}}
\right).
\end{align*}
Making a suitable change of variables, one can show that
$\widetilde{\boldsymbol{E}}(x;\delta,\boldsymbol{t})$ also has a logarithmic limiting distribution. In particular, for any Borel subset $B \subseteq \mathbb{R}^r,$ we have
\begin{align*}
\mathbb{P}_x^{\log}( \widetilde{\boldsymbol{E}}(x;\delta,\boldsymbol{t}) \in B )=\mu_{\delta,\boldsymbol{t}}(\widetilde{B}),
\end{align*}
where 
\begin{align*}
 \widetilde{B}:=\left\{ \boldsymbol{x}\in \mathbb{R}^r \,:\,\left( \frac{x_1}{\sqrt{V_1}},\ldots, \frac{x_r}{\sqrt{V_r}} \right) \in B \right\}.
\end{align*}

\begin{theorem} \label{thm:clt}
Assume RH and LI. Given a sufficiently small $\delta>0$ and an integer $r \geq 1$ for which $r/ \log \frac{1}{\delta}$ is sufficiently small,
let $T=T_{[r]} \leq \delta^{-\frac{1}{10}}$ and $B \subseteq \mathbb{R}^r$ be a Borel subset.
Then 
\begin{align*}
\mathbb{P}_x^{\log}(  \widetilde{\boldsymbol{E}}(x;\delta,\boldsymbol{t}) \in B )=&
\frac{1}{(2\pi)^{r/2}(\det\mathcal{C})^{1/2}} \int_{B}
\exp \left( -\frac{1}{2}\langle \mathcal{C}^{-1}\boldsymbol{x},\boldsymbol{x}\rangle  \right) d\boldsymbol{x}\\
&+O\left( r^4T^3 \left( \frac{1}{\sqrt{2\pi}}
+O \left(  \frac{\log 2r}{\log 1/\delta} \right) \right)^r  \left(\frac{1}{\delta}
\log \frac{1}{\delta} \right)^{-1} \meas(B) \right),
\end{align*}
where 
$\mathcal{C}=(c_{jk})_{r \times r}$ is the ``correlation matrix" with 
entries
\begin{align*}
c_{jk}=c_{jk}(\delta, \boldsymbol{t}):=
\frac{\Cov_{jk}}{\sqrt{V_j V_k}},
\end{align*}
i.e., the renormalized deviation of the weighted count of primes in multiple short intervals of length $\delta x$ is asymptotically normal with mean $\boldsymbol{0}$ and covariance matrix $\mathcal{C}.$
\end{theorem}

Although the error term here depends on the Borel set, the estimate can be made uniform over all such sets, provided that $r$ is small, i.e., there are not too many intervals.

\begin{theorem}\label{cor:compare}
Assume RH and LI. Given a sufficiently small $\delta>0$ and an
integer $1 \leq r \leq \frac{\log 1/\delta}{\log \log 1/\delta},$ let $T=T_{[r]} \leq \delta^{-\frac{1}{10}}$ and $B \subseteq \mathbb{R}^r$ be a Borel subset.
Then 
\begin{gather*}
\mathbb{P}_x^{\log}( \widetilde{\boldsymbol{E}}(x;\delta,\boldsymbol{t}) \in B )
=\frac{1}{(2\pi)^{r/2}(\det\mathcal{C})^{1/2}} \int_{B}
\exp \left( -\frac{1}{2}\langle \mathcal{C}^{-1}\boldsymbol{x},\boldsymbol{x}\rangle  \right) d\boldsymbol{x}\\
+O\left( \frac{r^4T^3}{(2\pi)^{r/2}}  \left(\frac{1}{\delta}
\log \frac{1}{\delta} \right)^{-1} 
\min \left\{\meas(B) , \left( 20\log \frac{1}{\delta} \right)^{r/2} \right\} \right).
\end{gather*}
In particular, the total variation distance satisfies 
\begin{align*}
\sup_{B \subseteq \mathbb{R}^r \,: \, B \, \text{\normalfont Borel}} 
\left| \mathbb{P}_x^{\log}( \widetilde{\boldsymbol{E}}(x;\delta,\boldsymbol{t}) \in B ) - \mathbb{P}(\mathcal{N}(\boldsymbol{0}, \mathcal{C}) \in B)  \right| \ll 
r^4 T^3 \delta \left( \frac{10}{\pi} \log \frac{1}{\delta}  \right)^{r/2-1},
\end{align*}
where $\mathcal{N}(\boldsymbol{0}, \mathcal{C})$ is an $r$-dimensional Gaussian random variable with mean $\boldsymbol{0}$ and covariance matrix $\mathcal{C}.$
\end{theorem}


As mentioned after Proposition~\ref{thm:cov}, the correlation between primes in two disjoint short intervals is asymptotically inversely proportional to their separation, and is negative. Therefore, informally speaking, the weighted prime counts in multiple short intervals resemble jointly normal point charges.

\begin{remark}
On the Fourier side,  as $t \in [T, 2T]$ varies for some large real number $T \geq 1$, the logarithm of the modulus of the Riemann zeta function on the critical line $\log |\zeta(\frac{1}{2}+it)|$ resembles a \textit{log-correlated Gaussian field} (see \cite[Theorem 1.1]{MR2678896}). In particular, if the separation between shifts $h_1,h_2$ satisfies  $\frac{1}{\log T} \ll |h_1-h_2|  \ll 1$ as $T \to \infty$, then the correlation between $\log |\zeta(\frac{1}{2}+it+ih_1)|$ and $\log |\zeta(\frac{1}{2}+it+ih_2)|$ is asymptotically proportional to $\log \left(|h_1-h_2|^{-1} \right),$ and is positive. It would then be natural to explore ``races" among $\log |\zeta(\frac{1}{2}+it+ih_j)|$ for many shifts $h_j$ (see the following discussions). 
\end{remark}

Many interesting consequences follow from Theorem~\ref{thm:clt} and Theorem~\ref{cor:compare}. First of all, if the number of intervals does not grow with \( \delta \) as \( \delta \to 0^+ \), then the negative correlations are explicitly detectable.



\begin{corollary}\label{cor:negcorr}
Assume RH and LI. 
Suppose both $r$ and $T=T_{[r]}$ are uniformly bounded with respect to a sufficiently small $\delta>0$. Let 
$R_{\boldsymbol{\alpha},\boldsymbol{\beta}}:=\prod_{i=1}^r(\alpha_i,\beta_i]$\label{not:Rab} be a box, where $\alpha_1,\ldots,\alpha_r,\beta_1,\ldots,\beta_r$ are real numbers or infinite.
Then 
\begin{gather*}
\mathbb{P}_x^{\log}( \widetilde{\boldsymbol{E}}(x;\delta,\boldsymbol{t}) \in R_{\boldsymbol{\alpha},\boldsymbol{\beta}} )
=\Phi(R_{\boldsymbol{\alpha},\boldsymbol{\beta}})
-\frac{1}{2\pi\log \frac{1}{\delta}} \sum_{1 \leq j<k \leq r} 
\Delta (|t_j-t_k|) \\
\cdot ( e^{ -\frac{1}{2}\alpha_j^2}-e^{ -\frac{1}{2}\beta_j^2})( e^{-\frac{1}{2}\alpha_k^2}-e^{-\frac{1}{2}\beta_k^2})
\Phi\left(\prod_{\substack{i=1\\i \neq j,k}}^r  (\alpha_i,\beta_i]\right)
+O_{r,T}\left( \frac{1}{\log^2 \frac{1}{\delta}} \right),
\end{gather*}
where
\begin{align*}
\Phi(R):=\frac{1}{(2\pi)^{r/2}}\int_{R} e^{-\frac{1}{2}\|\boldsymbol{x}\|^2} d\boldsymbol{x}.
\end{align*}
for measurable sets $R \subseteq \mathbb{R}^r.$
\end{corollary}
In particular, we have 
\begin{gather*}
\mathbb{P}_x^{\log}\left( E(x;\delta,t_1) , E(x;\delta,t_2), \ldots, E(x;\delta,t_r)>0 \right)
\\=
\frac{1}{2^r}-\frac{1}{2^{r-2}}\cdot \frac{1}{2\pi \log \frac{1}{\delta}}\sum_{1 \leq j <k \leq r} \Delta(|t_j-t_k|)
+O_{r,T} \left( \frac{1}{\log^2 \frac{1}{\delta}} \right),
\end{gather*}
and similarly
\begin{gather*}
\mathbb{P}_x^{\log}\left( E(x;\delta,t_1) , E(x;\delta,t_2), \ldots, E(x;\delta,t_r)<0 \right) \\=
\frac{1}{2^r}-\frac{1}{2^{r-2}}\cdot \frac{1}{2\pi \log \frac{1}{\delta}}\sum_{1 \leq j <k \leq r} \Delta(|t_j-t_k|) 
+O_{r,T} \left( \frac{1}{\log^2 \frac{1}{\delta}} \right),
\end{gather*}
i.e., it is less likely that the normalized deviations of weighted prime counts in short intervals are either all positive or all negative, as one might expect.

We also establish a large deviation estimate for the $L^2$-norm of prime count deviations. 

\begin{corollary} \label{cor:largedev}
Assume RH and LI. 
Suppose both $r$ and $T=T_{[r]}$ are uniformly bounded with respect to a sufficiently small $\delta>0$. Then 
\begin{align*}
\mathbb{P}_x^{\log}
( \|\widetilde{\boldsymbol{E}}(x;\delta,\boldsymbol{t})\| >V )
=\frac{1}{(2\pi)^{r/2}} \int_{\|\boldsymbol{x}\|>V}
\exp \left( -\frac{1}{2}\|\boldsymbol{x}\|^2 \right) d\boldsymbol{x}+O_{r,T}\left( \frac{1}{\log^2 \frac{1}{\delta}} \right).
\end{align*}
\end{corollary}

Motivated by classical results on primes in arithmetic progressions, Theorem~\ref{thm:clt} and Theorem~\ref{cor:compare} give rise to further applications.

In 1853, Chebyshev noted that on a fine scale there seem to be
more primes congruent to $3$ than to $1$ modulo $4$, now known as the \textit{Chebyshev's bias}. This observation led to the birth of \textit{comparative prime number theory}, which investigates the discrepancies in the distribution
of prime numbers (see \cite{MR146156} for an introduction).

A central problem is the \textit{Shanks--R\'{e}nyi prime number race} (see \cite{MR1985941} and \cite{MR2202918} for an introduction). 
Let $q \geq 3$ and $2 \leq r \leq \varphi(q)$ be positive integers, and denote by $\mathcal{A}_r(q)$ the set of
ordered $r$-tuples $(a_1,\ldots,a_r)$ of distinct residue classes that are coprime to $q$. 
Assuming the \textit{generalized Riemann hypothesis} (GRH) and the \textit{generalized linear independence conjecture} (GLI),\footnote{See \cite[p. 176]{MR1329368}, where GLI is referred to as the grand simplicity conjecture.} Rubinstein and Sarnak \cite{MR1329368} showed that for any  $(a_1,\ldots,a_r) \in \mathcal{A}_r(q),$
the inequality
\begin{gather*}
\pi(x;q,a_1)>\pi(x;q,a_2)>\cdots>\pi(x;q,a_r)
\end{gather*}
holds with a positive logarithmic density,\footnote{The logarithmic density of a set $S \subseteq (0,\infty)$, if it exists, is defined as
$\delta(S):=\lim_{X \to \infty} \frac{1}{\log X} \int_{S \cap [1,X]} \frac{dx}{x}.$ } denoted by $\delta(q;a_1,\ldots,a_r),$ where
\begin{align*}
\pi(x;q,a):=\#\{ p \leq x \,: \, p \equiv a \pmod{q} \}.
\end{align*}

For small moduli $q,$ as Chebyshev noticed, there are orderings of the $\pi(x;q,a_i)$'s not occurring with approximately
the same logarithmic density, which is $1/r!.$ 
Nevertheless, as $q \to \infty,$ Rubinstein and Sarnak \cite{MR1329368} proved, conditional on GRH and GLI, that any biases dissolve, provided the number of contestants $r$ does not grow with $q$. 


In this paper, we shall investigate the prime number race among many short intervals. To simplify notation, let us denote\label{not:rhodeltaboldsymbolt}
\begin{align*}
\rho(\delta;\boldsymbol{t}):=
\mathbb{P}_x^{\log}( \widetilde{E}(x;\delta,t_1)>
\widetilde{E}(x;\delta,t_2)>\cdots
>\widetilde{E}(x;\delta,t_r)
).
\end{align*}

When the number of intervals \( r \) does not grow as \( \delta \to 0^+ \), the following is a short-interval analog of \cite[Theorem 2.1]{MR3063909}.

\begin{corollary}\label{thm:extremebias}
Assume RH and LI. 
Suppose both $r$ and $T=T_{[r]}$ are uniformly bounded with respect to $\delta$. Then 
\begin{gather*}
\rho(\delta;\boldsymbol{t})=\frac{1}{r!}
-\frac{1}{\log \frac{1}{\delta}} \sum_{1 \leq j<k \leq r} 
\frac{\Delta (|t_j-t_k|)}{(2\pi)^{r/2}}
\int_{x_1>\cdots>x_r} x_jx_k e^{-\frac{1}{2}\|\boldsymbol{x}\|^2} d\boldsymbol{x}+O_{r,T}\left(\frac{1}{\log^2 \frac{1}{\delta}}\right).
\end{gather*}

\end{corollary}
Using Proposition \ref{thm:cov}, for each $j$, one can replace $\widetilde{E}(x;\delta,t_j)$ above by 
\begin{align*}
\psi(x;\delta,t_j):=\psi\left( (1+t_j \delta)x+\frac{1}{2}\delta x \right) -
\psi\left( (1+t_j\delta)x-\frac{1}{2}\delta x \right).
\end{align*}
For instance, in the case of three consecutive intervals, we have
\begin{align*}
\mathbb{P}_x^{\log}\left( \psi(x;\delta,\mp 1)>
\psi(x;\delta,0)
>\psi(x;\delta,\pm 1)
\right)&=
\frac{1}{6}-\frac{(\log4-\log3)\sqrt{3}}{4\pi\log \frac{1}{\delta}}+O\left(\frac{1}{\log^2\frac{1}{\delta}}\right)\\
&\approx 
\frac{1}{6}-\frac{0.039652}{\log \frac{1}{\delta}}.
\end{align*}
Meanwhile, the probability for each of the remaining four orderings is
\begin{align*}
\frac{1}{6}+\frac{(\log4-\log3)\sqrt{3}}{8\pi\log \frac{1}{\delta}}+O\left(\frac{1}{\log^2\frac{1}{\delta}}\right)
\approx  
\frac{1}{6}+\frac{0.019826}{\log \frac{1}{\delta}},
\end{align*}
i.e., it is less likely that the weighted counts of primes in three consecutive intervals (from left to right) appear in ascending or descending order.

The remaining corollaries describe a sharp phase transition from all races between primes in short intervals being asymptotically unbiased to the emergence of biased races.

When \( r = o\left(\frac{\log 1/\delta}{\log \log 1/\delta} \right) \)  as \( \delta \to 0^+ \), the following is a short-interval analog of \cite[Theorem~1.2]{MR3773805}.

\begin{corollary} \label{cor:bias}
Assume RH and LI. Given a sufficiently small $\delta>0$ and
an integer $2 \leq r \leq  \frac{\log 1/\delta}{\log \log 1/\delta},$ if $T_{[r]} \leq \delta^{-\frac{1}{10}},$ then 
\begin{align*}
\rho(\delta;\boldsymbol{t})=\frac{1}{r!}\left( 1+O\left( \frac{r\log r}{\log \frac{1}{\delta}} \right) \right).
\end{align*}
\end{corollary}

In particular, as $\delta \to 0^+,$ all $r$-way prime number races remain asymptotically unbiased, as long as $r=o\left( \frac{\log 1/\delta}{\log \log 1/\delta} \right).$

When $r \asymp \frac{\log 1/\delta}{\log \log 1/\delta}$ for any sufficiently small $\delta>0$, however, it turns out there exist $r$ intervals for which the corresponding race is noticeably biased (see \cite[Theorem 3]{MR3773805} and \cite[Theorem 1.2]{MR3961320} in the context of the Shanks--R\'{e}nyi prime number race). Following \cite{MR3773805} and \cite{MR3961320}, we first introduce the density
\begin{align*}
\rho_s(\delta;\boldsymbol{t}):=
\mathbb{P}_x^{\log} \left( 
\widetilde{E}(x;\delta,t_1)> \widetilde{E}(x;\delta,t_2)
>\cdots > \widetilde{E}(x;\delta,t_s)
>\max_{s < j \leq r} \widetilde{E}(x;\delta,t_j)
\right),
\end{align*}
which concerns the ordering of the first $s$ contestants in a race with $r$ competitors.

\begin{corollary} \label{cor:extremebias}

Assume RH and LI. 
Let $0<\varepsilon \leq 1.$ Then there exists $\eta=\eta(\varepsilon)>0$ such that the following holds. Given a sufficiently small $\delta>0$ and
integers $1 \leq  r \leq  \frac{\log 1/\delta}{\log \log 1/\delta}, 1 \leq s \leq \eta r,$
if $T_{[r]} \leq \delta^{-\frac{1}{10}},$  
then as soon as $s \to \infty$ with $\delta \to 0^+,$ we have 
\begin{align*}
\rho_s(\delta;\boldsymbol{t}) \leq \exp \left( o(1)-(2-\varepsilon) \cdot \frac{\log (r/s)}{\log \frac{1}{\delta}} \sum_{\substack{1 \leq j < k \leq s}}\Delta(|t_j-t_k|)\right) \frac{(r-s)!}{r!},
\end{align*}
\label{not:rhosdeltaboldsymbolt}provided that $|t_j-t_k|\geq \log \frac{1}{\delta}$ whenever $\max\{j,k \}>s$ for $1 \leq j \neq k \leq r.$ 
\end{corollary}

\begin{corollary} \label{cor:extremebiasii}
Assume RH and LI. Given a sufficiently small $\delta>0$ and an
integer $ \frac{\log 1/\delta}{\log \log 1/\delta} \ll r \leq  \frac{\log 1/\delta}{\log \log 1/\delta},$ there exists an absolute constant $\eta_0>0$ and $\boldsymbol{t} \in \mathbb{R}^r$ such that 
\begin{align*}
\rho(\delta;\boldsymbol{t}) \leq \exp \left( -\eta_0 \cdot  \frac{r\log \log \frac{1}{\delta}}{\log \frac{1}{\delta}} \right) \frac{1}{r!}.
\end{align*}
\end{corollary}

In contrast to \cite[Theorem 1.2]{MR3773805} and \cite[Theorem 1.2]{MR3961320}, where the phase transition lies within $[\log q/(\log \log q)^4, \log q],$ we can pinpoint the critical threshold in our setting, owing to the transparent correlation structure (see Lemma \ref{lem:eq:coulomb}) and refined estimates for the determinant and inverse of almost identity matrices (see Lemma \ref{lem:matrix}). This enables us to gain extra factors of $(\log \log \frac{1}{\delta})$ and $ (\log \log \frac{1}{\delta})^2 $, respectively, in order to generate a stronger bias by aggregating all the pairwise repulsions among the first $s$ intervals. It opens the gate to a sharper phase transition in the context of the Shanks--R\'{e}nyi prime number race, which we leave for future investigation.

Unfortunately, a Berry--Esseen type argument imposes a limitation on the number of intervals $r$ not exceeding $\frac{\log 1/\delta}{\log \log 1/\delta}.$ By employing either Stein’s method or the Lindeberg type method (see \cite{MR3773805} and \cite{MR3961320} respectively), it is certainly plausible to extend the range of $r$ up to a smaller power of $\delta^{-1},$ thus enabling the existence of extremely biased $r$-way prime number races. However, as our primary objective is to prove Theorems \ref{thm:clt} and \ref{cor:compare}, specifically in establishing a good upper bound on the total variation distance, we opt for the Fourier analytic approach. 





\section{Useful lemmas}
For convenience, we record here some properties of $w(s)$ and $\Delta(t)$, defined respectively in (\ref{eq:wsdt}) and (\ref{eq:biddelta}). 
\begin{lemma} \label{lem:eq:crude}
Let $\delta>0$ be sufficiently small. Then for $|x|\leq 1$ and $|y|>10,$ we have 
\begin{align*}
w(s)=w(x+iy) \leq 10 \cdot \min \left\{  y\delta, \frac{1}{|s|} \right\}.
\end{align*}
\begin{proof}
This follows from the definition of $w(s)$ and Taylor's theorem.
\end{proof}
\end{lemma}

\begin{lemma}  \label{lem:eq:coulomb}
For $t \geq 1,$ we have $\Delta(t) < 1.$ Also, 
as $t \to \infty,$ we have
\begin{align*}
\Delta(t)= \left(\frac{1}{2}+o(1) \right)t^{-1}.
\end{align*}
\begin{proof}
This follows from the definition of $\Delta(t)$ and Taylor's theorem.
\end{proof}
\end{lemma}

The following classical estimate for the number of non-trivial zeros is used throughout the paper without further comment.

\begin{lemma}[Riemann–von Mangoldt formula] \label{lem:riemvm}
Let $T \geq 2.$ Then
\begin{align*}
N(T)=\frac{T}{2\pi} \log \frac{T}{2\pi}-\frac{T}{2\pi}+O(\log T).
\end{align*}
\begin{proof}
See \cite[p. 98]{MR606931}.
\end{proof}
\end{lemma}

\section{Proof of Proposition \ref{thm:rs}}

While one may apply \cite[Theorem 4]{MR3261965}, which is a generalization of the work of Rubinstein and Sarnak \cite{MR1329368}, we provide a
detailed proof for the sake of completeness.


The following formula allows us to relate primes to zeros (under RH).
\begin{lemma}[Explicit formula] \label{lem:explicit}
Assume RH. Let $Z\geq 1.$ Then
\begin{align*}
E(x;\delta,t)
=-\sum_{|\gamma|\leq Z}w(\rho)x^{i\gamma}+O\left( \frac{\sqrt{x}\log^2 (xZ)}{Z}+\frac{\log x}{\sqrt{x}} \right).
\end{align*}
\begin{proof}
It follows immediately from the truncated von Mangoldt's explicit formula (see \cite[p. 109]{MR606931}) that
\begin{align*}
\psi(x)=x-\sum_{|\gamma|\leq Z}\frac{x^{\rho}}{\rho}
+O \left( \frac{x\log^2 (xZ)}{Z}+\log x \right)
\end{align*}
for $x \geq 2.$
\end{proof}
\end{lemma}

The following ergodic result allows us to invoke the LI hypothesis.
\begin{lemma}[Kronecker--Weyl] \label{lem:KW}
Let $\alpha_1,\dots,\alpha_N$ be linearly independent over $\mathbb{Q}$. Then the linear flow
\begin{align*}
t \mapsto (e^{2\pi i t \alpha_1}, \ldots, e^{2\pi i t \alpha_N})
\end{align*}
 for $t \geq 0$ is uniformly distributed over the torus $\mathbb{T}^N.$
\begin{proof}
See \cite[Exercise 9.27]{MR0419394}.
\end{proof}
\end{lemma}

Let $1 \leq Y \leq Z.$ Then by Lemma \ref{lem:explicit}, we have
\begin{align} \label{eq:truncate}
E(x;\delta,t)=E^{(Y)}(x;\delta,t)+\varepsilon^{(Y,Z)}(x;\delta,t),
\end{align}
where
\begin{align*}
E^{(Y)}(x;\delta,t):&=-\sum_{|\gamma|\leq Y}w(\rho)x^{i\gamma}\\
&=\Re\left(-2\sum_{0<\gamma \leq Y}\omega(\rho)x^{i \gamma}  \right)
\end{align*}
and
\begin{align*}
\varepsilon^{(Y,Z)}(x;\delta,t):=-\sum_{Y<|\gamma|\leq Z}w(\rho)x^{i\gamma}+O\left( \frac{\sqrt{x}\log^2 (xZ)}{Z}+\frac{\log x}{\sqrt{x}} \right).
\end{align*}
To lighten the notation, we write 
$\boldsymbol{E}(x)=\boldsymbol{E}^{(Y)}(x)+\boldsymbol{\varepsilon}^{(Y,Z)}(x)$ for $\boldsymbol{E}(x;\delta,\boldsymbol{t})=\boldsymbol{E}^{(Y)}(x;\delta,\boldsymbol{t})+\boldsymbol{\varepsilon}^{(Y,Z)}(x;\delta,\boldsymbol{t}).$

To prove Proposition \ref{thm:rs}, we require an $L^2$-norm bound on $\varepsilon^{(Y,Z)}(x;\delta,t).$

\begin{lemma} \label{lem:chebyshev}
Let $\delta>0$ be sufficiently small and $|t| \leq \delta^{-1}.$ Suppose $U, Y \geq 2$ with $Z=e^U \geq Y.$ Then
\begin{align*}
\int_{1}^U \left| {\varepsilon}^{(Y,Z)}(e^u;\delta,t) \right|^2 du
\ll 1+ \frac{U\log^2 Y}{Y}.
\end{align*}
\begin{proof}
The integral is 
\begin{gather*}
\ll \int_{1}^{U} \left| \sum_{Y<|\gamma| \leq Z}w(\rho)e^{i\gamma u} \right|^2 du
+e^{-U} U^5  + \int_1^U e^{-u} u^2  du \\
= \int_{1}^{U} \left| \sum_{Y<|\gamma| \leq Z}w(\rho)e^{i\gamma u} \right|^2 du +O(1).
\end{gather*}
Opening the square, the integral becomes
\begin{align*}
\sum_{Y<|\gamma|,|\gamma'|\leq Z}w(\rho)w(\rho')
\int_{1}^{U}e^{i(\gamma-\gamma')u}du
\ll \sum_{Y<|\gamma|,|\gamma'|\leq Z}\left|w(\rho)\right| \left|w(\rho')\right| \min \left\{ U,\frac{1}{|\gamma-\gamma'|} \right\}.
\end{align*}
Applying Lemma \ref{lem:eq:crude}, this is by partial summation
\begin{gather*}
\ll \int_Y^{\infty} \int_Y^{\infty} \frac{1}{xy}
\min \left\{ U,\frac{1}{|x-y|}  \right\}d N(x)d N(y)
\\
\ll \int_Y^{\infty} \int_Y^{\infty} \frac{\log x \log y}{xy}
\min \left\{ U,\frac{1}{|x-y|}  \right\}dxdy.
\end{gather*}
We split the integral into 
\begin{align}
I_1:=U{\iint}_{\substack{x,y \geq Y\\|x-y| \leq U^{-1}}} \frac{\log x \log y}{xy}dxdy
\ll& U\int_{x>Y} \frac{\log^2 x}{x^2} \left(\int_{\substack{y>Y\\ |x-y| \leq U^{-1}}} dy \right) dx \nonumber  \\
\ll& \frac{\log^2 Y}{Y}, \label{eq:int1fgs}
\end{align}
\begin{gather}
I_2:={\iint}_{\substack{x,y \geq Y\\U^{-1}<|x-y| \leq U^{-1}(Y/\log^2 Y) } }
\frac{\log x \log y}{xy}
\cdot \frac{ dxdy}{|x-y|}   \nonumber \\
\leq \sum_{1 \leq j \leq  Y/\log^2 Y}  
{\iint}_{\substack{x,y \geq Y\\jU^{-1}<|x-y| \leq (j+1)U^{-1} } }
\frac{\log x \log y}{xy}
\cdot \frac{dxdy}{|x-y|}  \nonumber \\
\ll \frac{\log^3 Y}{Y}, \label{eq:int2fgs}
\end{gather}
and
\begin{gather}
I_3:={\iint}_{\substack{x,y \geq Y\\|x-y|>U^{-1}(Y/\log^2 Y)}} 
\frac{\log x \log y}{xy}
\cdot \frac{dxdy}{|x-y|} \ll \frac{U\log^2 Y}{Y}. \label{eq:int3fgs}
\end{gather}
Since $U \geq \log Y,$ the lemma follows from combining (\ref{eq:int1fgs}), (\ref{eq:int2fgs}) and (\ref{eq:int3fgs}).
\end{proof}
\end{lemma}

\begin{proof}[Proof of Proposition \ref{thm:rs}]
Applying a theorem of Portmanteau \cite[Theorem 2.1]{MR0233396}, it suffices to verify Definition \ref{def} for all bounded Lipschitz continuous functions. Now let $f:\mathbb{R}^r \to \mathbb{R}$ be a bounded Lipschitz continuous function satisfying the inequality
\begin{align*}
|f(\boldsymbol{x})-f(\boldsymbol{y})| \leq c_f \|\boldsymbol{x}-\boldsymbol{y}\|
\end{align*}
for all $\boldsymbol{x},\boldsymbol{y} \in \mathbb{R}^r,$ where $c_f>0$ is the Lipschitz constant. Then for any $U \geq 1,$ we have
\begin{align*}
\frac{1}{U}\int_{1}^U f\left( \boldsymbol{E}(e^u) \right) du
&=\frac{1}{U}\int_{1}^U f\left( \boldsymbol{E}^{(Y)}(e^u)+\boldsymbol{\varepsilon}^{(Y,Z)}(e^u) \right) du\\
&=\frac{1}{U}\int_{1}^U f\left( \boldsymbol{E}^{(Y)}(e^u)\right) du
+O \left( \frac{c_f}{U} \int_{1}^U \| \boldsymbol{\varepsilon}^{(Y,Z)}(e^u) \| du\right) \\
&=\frac{1}{U}\int_{1}^U f\left( \boldsymbol{E}^{(Y)}(e^u)\right) du 
+ O \left( \frac{c_f}{U} \sum_{j=1}^r \int_{1}^U | {\varepsilon}^{(Y,Z)}(e^u;\delta,t_j) | du  \right).
\end{align*}
Let $N=N(Y):=\#\{0 <\gamma \leq Y\}$ and list the zeros as $0<\gamma_1<\cdots<\gamma_N\leq Y$ (simplicity of zeros is guaranteed by LI). Then, we define the function $\varphi_N:\mathbb{T}^N \to \mathbb{R}^r$ by
\begin{align*}
\varphi_N\left( e^{2\pi i \boldsymbol{\theta}} \right)
:= \left(\Re \left( -2\sum_{n \leq N} w_1(\rho_n)e^{2\pi i \theta_n} \right),\dots, \Re \left( -2\sum_{n \leq N} w_r(\rho_n)e^{2\pi i \theta_n} \right) \right),
\end{align*}
where $e^{2\pi i \boldsymbol{\theta}}=\left( e^{2\pi i \theta_1}, \ldots, e^{2\pi i \theta_N} \right)$ and $ \rho_n=\frac{1}{2}+i\gamma_n$ for $n=1,\ldots,N.$
By definition, we have
\begin{align*}
\boldsymbol{E}^{(Y)}(e^u)  &=  
\left(\Re \left( -2\sum_{n \leq N} w_1(\rho_n)e^{i \gamma_n u} \right),\dots, \Re \left( -2\sum_{n \leq N} w_r(\rho_n)e^{i \gamma_n u} \right) \right)\\
&=\varphi_N \left( 
e^{i\gamma_1u},\ldots,e^{i\gamma_Nu}
\right).
\end{align*}
The assumption of LI implies that $\gamma_1/2\pi,\dots,\gamma_N/2\pi$ are linearly independent over $\mathbb{Q},$ and hence it follows from Lemma \ref{lem:KW} that for any fixed $Y \geq 2,$ we have
\begin{align*}
\underset{U \to \infty}{\lim}\frac{1}{U}\int_{1}^U f\left( \boldsymbol{E}^{(Y)}(e^u)\right) du
=\int_{[0,1]^N} \left(f \circ \varphi_N \right) \left( e^{2\pi i \boldsymbol{\theta}}\right)
d\boldsymbol{\theta}.
\end{align*}
On the other hand, the Cauchy--Schwarz inequality yields
\begin{align*}
\frac{1}{U}\int_{1}^U | \varepsilon^{(Y,Z)}(e^u;\delta,t_j) | du \leq 
\frac{1}{\sqrt{U}}
\left(\int_{1}^U | \varepsilon^{(Y,Z)}(e^u;\delta,t_j) |^2 du\right)^{\frac{1}{2}}
\end{align*}
for $j=1,\ldots,r$. Therefore, it follows from Lemma \ref{lem:chebyshev} that
\begin{align*}
\underset{U \to \infty}{\limsup} \,
\frac{1}{U}\int_{1}^U f\left( \boldsymbol{E}(e^u) \right) du
=\int_{[0,1]^N} \left(f \circ \varphi_N \right)  \left( e^{2\pi i \boldsymbol{\theta}}\right)
d\boldsymbol{\theta}+O\left( \frac{c_fr\log Y}{\sqrt{Y}} \right),
\end{align*}
and similarly,
\begin{align*}
\underset{U \to \infty}{\liminf} \,
\frac{1}{U}\int_{1}^U f\left( \boldsymbol{E}(e^u) \right) du
=\int_{[0,1]^N} \left(f \circ \varphi_N \right)  \left( e^{2\pi i \boldsymbol{\theta}}\right)
d\boldsymbol{\theta}+O\left( \frac{c_fr\log Y}{\sqrt{Y}} \right),
\end{align*}
Taking $Y \to \infty, $ we establish the existence of the limit
\begin{align} \label{eq:exist}
 \underset{U \to \infty}{\lim}
\frac{1}{U}\int_{1}^U f\left( \boldsymbol{E}(e^u) \right) du
=\lim_{N \to \infty}\int_{[0,1]^N} \left(f \circ \varphi_N \right)  \left( e^{2\pi i \boldsymbol{\theta}}\right)
d\boldsymbol{\theta},
\end{align}
which is $\underset{Y \to \infty}{\lim}\mathbb{E}(f(\boldsymbol{X}_{\delta, \boldsymbol{t}}^{(Y)}) ),$ where
 $\boldsymbol{X}_{\delta, \boldsymbol{t}}^{(Y)}$ is the truncated random vector
\begin{align*}
\left( \Re \left(2 \sum_{0<\gamma \leq Y} w_1(\rho)U_{\gamma}\right), \ldots, \Re \left(2 \sum_{0<\gamma \leq Y} w_r(\rho)U_{\gamma}\right) \right).
\end{align*}

Let $F_Y$ denote the corresponding distribution function, i.e., $F_Y(\boldsymbol{x}):=\mathbb{P}( \boldsymbol{X}_{\delta, \boldsymbol{t}}^{(Y)} \leq \boldsymbol{x} )$ for $\boldsymbol{x} \in \mathbb{R}^r.$
Then by Helly's selection principle, there exists an increasing subsequence
$\{Y_k\}$ such that $F_{Y_k}$ converges weakly to a generalized distribution function (increasing, right-continuous) $F_{\delta,\boldsymbol{t}}$ as $k \to \infty.$ In particular, it follows from (\ref{eq:exist}) that
\begin{align*}
\underset{U \to \infty}{\lim}
\frac{1}{U}\int_{1}^U f\left( \boldsymbol{E}(e^u) \right) du
=\int_{\mathbb{R}^r} f(\boldsymbol{x})dF_{\delta,\boldsymbol{t}}(\boldsymbol{x}).
\end{align*}
Taking $f \equiv 1,$ we conclude that $F_{\delta,\boldsymbol{t}}$ is a proper distribution function, i.e., 
\begin{align*}
\lim_{x_1,\ldots,x_r \to +\infty}F_{\delta,\boldsymbol{t}}(x_1,\ldots,x_r)=1.
\end{align*}

Finally, let $\mu_{\delta,\boldsymbol{t}}$ denote the probability measure induced by $F_{\delta,\boldsymbol{t}}$. In order to establish the correspondence with $\boldsymbol{X}_{\delta,\boldsymbol{t}},$ we apply L\'{e}vy's continuity theorem with Lemma \ref{lem:bessel}, which ensures that the characteristic functions of $\mu_{\delta,\boldsymbol{t}}$ and $\boldsymbol{X}_{\delta,\boldsymbol{t}}$ coincide. Therefore, the proposition follows.
\end{proof}

\section{Proof of Proposition \ref{thm:cov}}

Let $1 \leq j,k \leq r$ be fixed. 
Using the fact that
\begin{align*}
\frac{1}{\rho}=\frac{1}{\frac{1}{2}+i\gamma}=\frac{1}{i\gamma} \left( 1+O\left( \frac{1}{|\gamma|} \right) \right)
\end{align*}
and the assumption $t_j \leq T=T_{\{j,k\}} \leq \delta^{-\frac{1}{2}}\left( \log \frac{1}{\delta}\right)^{-1},$ so that
\begin{align*}
\left( 1+\left( t_j \pm \frac{1}{2} \right)\delta \right)^{\frac{1}{2}}=1+O\left( T\delta \right),
\end{align*}
we have
\begin{gather*} 
w_j(\rho)=\frac{1}{\rho}\left( \left( 
1+\left( t_j+\frac{1}{2} \right)\delta
\right)^{\rho}-\left( 1+\left( t_j-\frac{1}{2} \right)\delta \right)^{\rho} \right) \\
=\frac{1}{i\gamma}\left( 1+O\left( \frac{1}{|\gamma|} \right) \right)\left( \left( 1+O\left( T\delta \right) \right)\left( 
1+\left( t_j+\frac{1}{2} \right)\delta
\right)^{i\gamma} \right.\\
\left.-\left( 1+O\left( T\delta \right) \right)\left( 
1+\left( t_j-\frac{1}{2} \right)\delta
\right)^{i\gamma} \right)\\
=w_j(i\gamma)\left( 1+O\left(\frac{1}{|\gamma|} \right) \right)+O\left( \frac{T\delta}{|\gamma|} \right).
\end{gather*}
Similarly, we have
\begin{align*}
w_k(\rho)=w_k(i\gamma)\left( 1+O\left(\frac{1}{|\gamma|} \right) \right)+O\left( \frac{T\delta}{|\gamma|} \right).
\end{align*}
Multiplying the last two expressions together gives
\begin{gather*}
w_j(\rho)\overline{w_k(\rho)}=
w_j(i\gamma)w_k(-i\gamma)\left( 1+O\left( \frac{1}{|\gamma|} \right) \right) \\
+O\left( \frac{T\delta}{|\gamma|}(|w_j(i\gamma)|+|w_k(-i\gamma)|) \right)+O\left( \left(\frac{T\delta}{|\gamma|}\right)^2 \right).
\end{gather*}
Summing over $\gamma,$ we have
\begin{gather}
\sum_{\gamma}w_j(\rho)\overline{w_k(\rho)}
=\sum_{\gamma}w_j(i\gamma)w_k(-i\gamma)
+O \left( \sum_{\gamma} \frac{\left|w_j(i\gamma)w_k(-i\gamma)\right|}{|\gamma|} \right) \nonumber\\
+O\left(T\delta\sum_{\gamma}\frac{|w_j(i\gamma)|+|w_k(i\gamma)|}{|\gamma|}\right)
+O\left( T^2\delta^2 \sum_{\gamma}\frac{1}{|\gamma|^2} \right). \label{eq:3errors}
\end{gather}
Applying Lemma \ref{lem:eq:crude},  
the contribution of $|\gamma| \leq \frac{1}{T\delta}$ to the first error term is
\begin{align} \label{eq:et11}
\ll \sum_{|\gamma| \leq \frac{1}{T \delta}} \frac{\left|w_j(i\gamma)w_k(-i\gamma)\right|}{|\gamma|} 
&\ll T^2\delta^2
\sum_{|\gamma| \leq \frac{1}{T \delta}} \frac{1}{|\gamma|}\nonumber \\
&\ll T^2\delta^2 \int_{1}^{1/T\delta} \frac{\log t}{t}dt
\nonumber\\
&\ll T^2\delta^2 \log^2 \frac{1}{T \delta},
\end{align}
and the contribution of $|\gamma| > \frac{1}{T\delta}$ is
\begin{align} \label{eq:et12}
\ll \sum_{|\gamma| > \frac{1}{T \delta}} \frac{\left|w_j(i\gamma)w_k(-i\gamma)\right|}{|\gamma|}
&\ll \sum_{|\gamma| > \frac{1}{T \delta}} \frac{1}{|\gamma|^3} \nonumber\\
&\ll \int_{1/T\delta}^{\infty} \frac{\log t}{t^3} dt \nonumber\\
&\ll T^2\delta^2 \log \frac{1}{T \delta}.
\end{align}
Similarly, the second error term in (\ref{eq:3errors}) is
$\ll T^2\delta^2 \log^2 \frac{1}{T\delta},$ and the third error term is $\ll T^2 \delta^2.$

We write $w_j(i\gamma)w_k(-i\gamma)$ as
\begin{align} \label{eq:euler}
\frac{4}{\gamma^2}
\exp \left( \frac{1}{2}i \left( \lambda_j^+ + \lambda_j^- -  \lambda_k^+ - \lambda_k^-\right) \gamma \right)
\sin \left( \frac{1}{2}(\lambda_j^+ - \lambda_j^-)\gamma\right)
\sin \left( \frac{1}{2}(\lambda_k^+ - \lambda_k^-)\gamma\right),
\end{align}
where 
\begin{align*}
\lambda_j^{\pm}=\log \left( 1+\left( t_j \pm \frac{1}{2} \right)\delta \right), \quad
\lambda_k^{\pm}=\log \left( 1+\left( t_k \pm \frac{1}{2} \right)\delta \right).
\end{align*}
By Taylor expansion, we have
\begin{align} \label{eq:lambda1}
\lambda_j^+ + \lambda_j^- -  \lambda_k^+ - \lambda_k^-
=2 (t_j-t_k)\delta +O(\left(T\delta\right)^2) 
\end{align}
and
\begin{align} \label{eq:lambda2}
\lambda_j^+ - \lambda_j^-=\delta+O\left(T\delta^2\right) , \quad
\lambda_k^+ - \lambda_k^-=\delta+O\left(T\delta^2\right) .
\end{align}
Recall that 
\begin{align} \label{eq:real}
\sum_{\gamma}w_j(i\gamma)w_k(-i\gamma)=2\sum_{\gamma>0} \Re \left( w_j(i\gamma)w_k(-i\gamma) \right).
\end{align}
Substituting (\ref{eq:lambda1}) and (\ref{eq:lambda2}) into (\ref{eq:euler}), we have
\begin{align*}
\Re \left( w_j(i\gamma)w_k(-i\gamma) \right)&=\frac{4}{\gamma^2}
\cos \left(  \left( t_j-t_k \right) \delta \gamma +O \left( T^2\delta^2 |\gamma | \right) \right) \sin^2 \left( \frac{1}{2} \delta \gamma
+ O \left(  T \delta^2 |\gamma |\right) \right)\\
&=\frac{4}{\gamma^2}
\cos \left(  \left( t_j-t_k \right) \delta \gamma \right) \sin^2 \left( \frac{1}{2} \delta \gamma \right)+O\left( \min \left\{ \frac{1}{|\gamma|^2}, \frac{T^2\delta^2}{|\gamma|}  \right\} \right)
\end{align*}
using elementary trigonometric identities followed by Taylor expansion.

The contribution of the last error term to (\ref{eq:real}) is
\begin{align} \label{eq:et13}
\ll \sum_{\gamma>0}\frac{\min \{ 1, T^2\delta^2 \gamma \}}{\gamma^2}
&\ll T^2\delta^2 \sum_{0<\gamma \leq \frac{1}{T^2\delta^2}} \frac{1}{\gamma}
+\sum_{\gamma>\frac{1}{T^2\delta^2}}\frac{1}{\gamma^2} \nonumber\\
&\ll T^2\delta^2 \log^2 \frac{1}{T\delta}.
\end{align}
Therefore, we are left with estimating the sum
\begin{align*}
8\sum_{\gamma>0}\frac{\cos \left(  \left( t_j-t_k \right) \delta \gamma \right)}{\gamma^2}
 \sin^2 \left( \frac{1}{2} \delta \gamma \right).
\end{align*}
Using the trigonometric identity
\begin{align*}
\cos(2x \theta)\sin^2 (\theta)=
\frac{1}{2}\left( \sin^2((x+1)\theta)-2\sin^2(x \theta)+\sin^2((x-1)\theta) \right)
\end{align*}
with
$\theta=\frac{1}{2}\delta \gamma$ and $x=t_j-t_k$,
it remains to evaluate the sum
\begin{align*} 
4\sum_{\gamma>0}\frac{\sin^2 \left( y \delta \gamma \right /2)}{\gamma^2}
\end{align*}
for $y=t_j-t_k+1, t_j-t_k$ and $t_j-t_k-1.$

If $y=0,$ then the sum simply vanishes. Otherwise, suppose $y \neq 0.$ Then it can be expressed as the Riemann--Stieltjes integral
\begin{align} \label{eq:rsintegral}
4\int_{0}^{\infty} \frac{\sin^2 \left( y \delta t \right /2)}{t^2} dN(t).
\end{align}
Applying Lemma \ref{lem:riemvm}, we obtain
\begin{align*}
dN(t)=\frac{1}{2\pi}\log \left( \frac{t}{2\pi} \right)dt+dE(t),
\end{align*}
where $E(t) \ll \log (t+2).$ Integrating by parts, we have
\begin{align*}
\int_{0}^{\infty} \frac{\sin^2 \left( y \delta t \right /2)}{t^2} dE(t)&=\left[ \frac{\sin^2 \left( y \delta t \right /2)}{t^2} E(t)  \right]_{t=0^+}^{\infty}-\int_{0}^{\infty} \frac{d}{dt}\left(\frac{\sin^2 \left( y \delta t \right /2)}{t^2}\right) E(t) dt\\
&=-\int_{0}^{\infty} \frac{d}{dt}\left(\frac{\sin^2 \left( y \delta t \right /2)}{t^2}\right) E(t) dt.
\end{align*}
The quotient rule gives
\begin{align*}
\frac{d}{dt}\left(\frac{\sin^2 \left( y \delta t \right /2)}{t^2}\right) =
 \frac{y \delta t^2 \sin (y \delta t/2)\cos(y \delta t/2)-2t\sin^2 (y \delta t/2) }{t^4}
\end{align*}
and since $E(t)=0$ for $t \leq 1,$ we have
\begin{gather*}
\int_{0}^{\infty} \frac{d}{dt}\left(\frac{\sin^2 \left( y\delta t \right /2)}{t^2}\right) E(t) dt \\ 
\ll y \delta \int_{1}^{\infty} |\sin(y \delta t)| \frac{\log (t+2)}{t^2} dt+ \int_{1}^{\infty} \sin^2 \left( \frac{1}{2}y \delta t \right) \frac{\log (t+2)}{t^3} dt.
\end{gather*}
Using the fact that
$\sin x \ll \min\{|x|,1 \}$ and $|y| \leq 2T,$
this is
\begin{align} \label{eq:et14}
\ll y^2 \delta^2 \log \frac{1}{|y \delta|} \ll T^2\delta^2 \log \frac{1}{T \delta}.
\end{align}

On the other hand, the main term of (\ref{eq:rsintegral}) is
\begin{align*}
\frac{2}{\pi}\int_{0}^{\infty}\frac{\sin^2 \left( y \delta t \right /2)}{t^2} \log \left( \frac{t}{2\pi} \right) dt=
\frac{1}{\pi^2}\int_{0}^{\infty}\frac{\sin^2 (\pi y \delta x)}{x^2}\log x dx.
\end{align*}
Collecting the error terms in (\ref{eq:et11}), (\ref{eq:et12}), (\ref{eq:et13}) and (\ref{eq:et14}), the proposition follows from the next lemma with $\kappa=\pi y \delta$ and the assumption that $|t_j-t_k| \geq 1$ for $1 \leq j \neq k \leq r.$
\begin{lemma}
Let $\kappa \in \mathbb{R}.$ Then 
\begin{align*}
\int_{0}^{\infty} \left( \frac{\sin \kappa x}{x} \right)^2
\log x dx
=\frac{\pi}{2}|\kappa|(1-\gamma_{\mathbb{Q}}-\log 2|\kappa|).
\end{align*}
\begin{proof}
See \cite[p. 598]{MR3307944}.
\end{proof}
\end{lemma}

\section{Proof of Theorem \ref{thm:clt}} \label{section5}


Throughout the section, we always assume both
$\delta$ and $ r/\log(1/\delta)$ are sufficiently small,
$T=T_{[r]} \leq \delta^{-\frac{1}{10}},$ and $|t_j-t_k|\geq 1$ for $1 \leq j \neq k \leq r$ as in the statement of Theorem \ref{thm:clt}.

To prove Theorem \ref{thm:clt}, we made adjustments to several lemmas from \cite{MR2998146} and \cite{MR3063909}. Let us first establish several properties of the Fourier transform \( \hat{\mu}_{\delta, \boldsymbol{t}} \), defined below, which are essential to our proof.

\begin{lemma} \label{lem:bessel}
Let $\hat{\mu}_{\delta,\boldsymbol{t}}$ denote the Fourier transform of the measure $\mu_{\delta;\boldsymbol{t}},$ i.e.,
\begin{align*}
\hat{\mu}_{\delta,\boldsymbol{t}}\left( \boldsymbol{\xi} \right)
:=\int_{\mathbb{R}^r}e^{-i\langle \boldsymbol{x},\boldsymbol{\xi}\rangle}d\mu_{\delta;\boldsymbol{t}}(\boldsymbol{x})
\end{align*}
for $\boldsymbol{\xi} \in \mathbb{R}^r.$ Then 
\begin{align*}
\hat{\mu}_{\delta,\boldsymbol{t}}(\boldsymbol{\xi})=\prod_{\gamma>0}
J_0 \left( 2\left| \sum_{j=1}^r w_j(\rho)\xi_j \right| \right),
\end{align*}
where 
\begin{align*}
J_0(x):=\frac{1}{2\pi}\int_{0}^{2\pi}e^{ix\cos(\theta)} d\theta
\end{align*}
is the Bessel function of order $0$. 
\begin{proof}
By the proof of Proposition \ref{thm:rs}, we have
\begin{align} \label{eq:easycov}
\hat{\mu}_{\delta,\boldsymbol{t}}(\boldsymbol{\xi})&=\lim_{Y \to \infty}\mathbb{E}\left(\exp\left(-i\langle \boldsymbol{X}^{(Y)}_{\delta,\boldsymbol{t}},\boldsymbol{\xi} \rangle \right)\right) \nonumber\\
&=\prod_{\gamma>0}\mathbb{E}\left( \exp \left( -i \Re \left(2\sum_{j=1}^r w_j(\rho) U_{\gamma}\cdot \xi_j \right)\right) \right) \nonumber\\
&=\prod_{\gamma>0}\int_{0}^{1} \exp  
\left( 2i\Re \left(- \left(\sum_{j=1}^r w_j(\rho)\xi_j\right)e^{2\pi i \theta}\right) \right) d \theta.
\end{align}
For each $\gamma>0,$ we write
\begin{align*}
-\sum_{j=1}^r w_j(\rho)\xi_j =  \left| \sum_{j=1}^r w_j(\rho)\xi_j \right| e^{2\pi i \theta_{\gamma}}
\end{align*}
for some $\theta_{\gamma} \in [0,1).$ 
Then making the change of variables $\varphi=\theta+\theta_{\gamma}$ for $\gamma>0$, the expression (\ref{eq:easycov}) becomes
\begin{align*}
\prod_{\gamma>0}\int_{0}^{1}\exp  
\left( 2i\left| \sum_{j=1}^r w_j(\rho)\xi_j\right| 
\cos \left( 2\pi \varphi \right) \right) d \varphi
=\prod_{\gamma>0} J_0 \left( 2\left| \sum_{j=1}^r w_j(\rho)\xi_j \right| \right),
\end{align*}
and hence the lemma follows.
\end{proof}
\end{lemma}

For convenience, we record some properties of the Bessel function $J_0(x).$

\begin{lemma} \label{lem:besselprop}
The Bessel function $J_0(x)$ satisfies the following properties.
\begin{itemize}
    \item Let $x \in \mathbb{R}.$ Then
\begin{align} \label{eq:bessel1}
|J_0(x)| \leq \min \left\{1,\sqrt{\frac{2}{\pi|x|}} \right\}
\end{align} 
\item Let $x \in [0,1].$ Then
\begin{align} \label{eq:bessel2}
 \max_{y \geq x}J_0(y) = J_0(x).
\end{align}
\item Let $|x| \leq 1.$ Then
\begin{align} \label{eq:bessel3}
|J_0(x)| \leq \exp(-x^2/4).
\end{align}
\item Let $|x| \leq 1.$ Then
\begin{align} \label{eq:bessel4}
\log J_0(x)=-\frac{1}{4}x^2+O(x^4).
\end{align}
\end{itemize}

\begin{proof}
See \cite[Chapter VII]{MR10746}.
\end{proof}

\end{lemma}

To bound the Fourier transform \( \hat{\mu}_{\delta, \boldsymbol{t}} \), we require the following lemma.

\begin{lemma} \label{lem:largespec}
We define the large spectrum by
\begin{align*}
\mathcal{M}_{\delta}(\boldsymbol{\xi}):=\left\{ 0<\gamma \leq \frac{1}{\delta}\log \frac{1}{\delta}\, : \, 
\left| \sum_{j=1}^r w_j(\rho)\xi_j\right|>\frac{1}{2}\left(\frac{1}{\delta}\log \frac{1}{\delta}
\right)^{-1}\|\boldsymbol{\xi}\|\right\}.
\end{align*}
Then 
\begin{align*}
\left| \mathcal{M}_{\delta}(\boldsymbol{\xi}) \right| > \frac{1}{250rT^2}\left(\frac{1}{\delta}
\log \frac{1}{\delta}\right).
\end{align*}
\begin{proof}
We define
\begin{align*}
S_{\delta}(\boldsymbol{\xi}):=\sum_{0<\gamma \leq \frac{1}{\delta}\log \frac{1}{\delta}}
\left| \sum_{j=1}^r w_j(\rho)\xi_j \right|^2.
\end{align*}
Opening the square, this becomes
\begin{equation}
\begin{gathered}[b]\label{eq:l2}
\sum_{1 \leq j,k, \leq r}
\left( \sum_{0 < \gamma \leq \frac{1}{\delta}\log \frac{1}{\delta}}w_j(\rho) \overline{w_k(\rho)} \right)
\xi_j \xi_k \\
=\sum_{j=1}^r \left( \sum_{0 < \gamma \leq \frac{1}{\delta}\log \frac{1}{\delta}}
|w_j(\rho)|^2\right)|\xi_j|^2+\sum_{1 \leq j \neq k \leq r} \left( \sum_{0<\gamma \leq \frac{1}{\delta}\log \frac{1}{\delta}} \Re \left(w_j(\rho)\overline{w_k(\rho)}\right) \right)\xi_j \xi_k.
\end{gathered}
\end{equation}
By the proof of Proposition \ref{thm:cov}, we have
\begin{align}\label{eq:big}
\sum_{\gamma>\frac{1}{\delta}\log \frac{1}{\delta}}|w_j(\rho)|^2
&=\frac{2}{\pi}\int_{\frac{1}{\delta}\log \frac{1}{\delta}}^{\infty} \left( \frac{\sin\left( \delta t/2 \right)}{t} \right)^2\log \frac{t}{2\pi} dt 
+O(\delta).
\end{align}
Making the change of variables $x=\delta t/2,$ the main term of the right-hand side becomes
\begin{align*}
\frac{1}{\pi}\delta \int_{\frac{1}{2}\log \frac{1}{\delta}}^{\infty} 
\left( \frac{\sin x}{x} \right)^2 \log \left(\frac{x}{\pi \delta}\right) dx \ll \delta.
\end{align*}
Combining (\ref{eq:big}) with Proposition \ref{thm:cov}, we have
\begin{align*}
\sum_{0 < \gamma \leq \frac{1}{\delta}\log \frac{1}{\delta}} |w_j(\rho)|^2 =
\frac{1}{2}\delta \log \frac{1}{\delta}+O(\delta).
\end{align*}
Similarly, one can show that
\begin{align*}
\sum_{0<\gamma \leq \frac{1}{\delta}\log \frac{1}{\delta}} \Re ( w_j(\rho)\overline{w_k(\rho)} ) \ll \delta.
\end{align*}
Therefore, the expression (\ref{eq:l2}) becomes 
\begin{align}\label{eq:compare}
S_{\delta}(\boldsymbol{\xi}) =\left( \frac{1}{2}\delta \log \frac{1}{\delta}+O\left(r \delta \right) \right)\|\boldsymbol{\xi}\|^2. 
\end{align}

On the other hand, we split $S_{\delta}(\boldsymbol{\xi})$  into
\begin{align*}
\sum_{\gamma \in \mathcal{M}_{\delta}(\boldsymbol{\xi})}
\left| \sum_{j=1}^r w_j(\rho)\xi_j \right|^2
+\sum_{\gamma \in \left( 0,\frac{1}{\delta}\log \frac{1}{\delta} \right] \setminus\mathcal{M}_{\delta}(\boldsymbol{\xi})} \left| \sum_{j=1}^r w_j(\rho)\xi_j \right|^2.
\end{align*}
Applying Lemma \ref{lem:eq:crude} followed by the Cauchy--Schwarz inequality, the sum over the large spectrum is
\begin{align} \label{eq:largespec}
\sum_{\gamma \in \mathcal{M}_{\delta}(\boldsymbol{\xi})}
\left| \sum_{j=1}^r w_j(\rho)\xi_j \right|^2
\leq (10T\delta)^2 r\left|\mathcal{M}_{\delta}(\boldsymbol{\xi})\right| \| \boldsymbol{\xi} \|^2.
\end{align}
By the definition of the large spectrum, the remaining sum is
\begin{align} \label{eq:remain}
\sum_{\gamma \in \left( 0,\frac{1}{\delta}\log \frac{1}{\delta} \right] \setminus\mathcal{M}_{\delta}(\boldsymbol{\xi})} \left| \sum_{j=1}^r w_j(\rho)\xi_j \right|^2
&\leq N \left( \frac{1}{\delta}\log \frac{1}{\delta} \right) 
\left(  \frac{1}{2}\left( \frac{1}{\delta}\log \frac{1}{\delta} \right)^{-1}\|\boldsymbol{\xi}\| \right)^2 \nonumber\\
& \leq \frac{1}{4\pi}\delta \|\boldsymbol{\xi}\|^2.
\end{align}
Combining (\ref{eq:compare}), (\ref{eq:largespec}) and (\ref{eq:remain}), we obtain
\begin{align*}
\left| \mathcal{M}_{\delta}(\boldsymbol{\xi}) \right| \geq \frac{1}{200rT^2}\left(\frac{1}{\delta}
\log \frac{1}{\delta}\right)+O\left( \frac{1}{T^2\delta} \right),
\end{align*}
and the lemma follows from our assumptions on $\delta, r.$
\end{proof}
\end{lemma}

\begin{lemma}\label{lem:loc}
Let $0<\varepsilon \leq 1.$ Then 
\begin{align*}
\hat{\mu}_{\delta,\boldsymbol{t}}\left(\left( \frac{1}{\delta} \log \frac{1}{\delta}\right)\boldsymbol{\xi}\right)\ll
\begin{cases}
\hfil \exp \left( -\dfrac{\log \|\boldsymbol{\xi}\|}{500rT^2} \left(\dfrac{1}{\delta}\log \dfrac{1}{\delta} \right)\right)  & \mbox{{\normalfont if $\|\boldsymbol{\xi}\|>2,$ } }\\
\exp \left( -\dfrac{\varepsilon^2}{1000rT^2} \left(\dfrac{1}{\delta}\log \dfrac{1}{\delta} \right) \right) & \mbox{{\normalfont if $\varepsilon < \|\boldsymbol{\xi}\| \leq 2.$ } }
\end{cases}
\end{align*}
\begin{proof}
Suppose $\|\boldsymbol{\xi}\|>2.$ Since $|J_0(x)| \leq 1$ for $x \in \mathbb{R},$ we have
\begin{align*}
\left|\hat{\mu}_{\delta,\boldsymbol{t}}\left(\left(\frac{1}{\delta}\log \frac{1}{\delta}\right)\boldsymbol{\xi}\right)\right|&=\prod_{\gamma>0}\left|J_0 \left( 2\left(\frac{1}{\delta}\log \frac{1}{\delta}\right)
\left| \sum_{j=1}^r w_j(\rho)\xi_j \right| \right)\right|\\
&\leq \prod_{\gamma \in \mathcal{M}_{\delta}(\boldsymbol{\xi})}
\left|J_0 \left(  2\left(\frac{1}{\delta}\log \frac{1}{\delta}\right)
\left| \sum_{j=1}^r w_j(\rho)\xi_j \right| \right)\right|.
\end{align*}
Using (\ref{eq:bessel1}) of Lemma \ref{lem:besselprop}, the product is bounded by
\begin{align*}
\prod_{\gamma \in \mathcal{M}_{\delta}(\boldsymbol{\xi})}
\sqrt{\frac{2}{\pi\|\boldsymbol{\xi}\|}}
\leq \|\boldsymbol{\xi}\|^{-\frac{1}{2}|\mathcal{M}_{\delta}(\boldsymbol{\xi})|}.
\end{align*}

On the other hand, suppose $\varepsilon<\|\boldsymbol{\xi}\|\leq 2.$ Using (\ref{eq:bessel2}) of Lemma \ref{lem:besselprop}, we have
\begin{align*}
\left|\hat{\mu}_{\delta,\boldsymbol{t}}\left(\left( \frac{1}{\delta} \log \frac{1}{\delta} \right)\boldsymbol{\xi}\right)\right|
&\leq \prod_{\gamma \in \mathcal{M}_{\delta}(\boldsymbol{\xi})}
\left|J_0 \left(  2\left(\frac{1}{\delta}\log \frac{1}{\delta}\right)
\left| \sum_{j=1}^r w_j(\rho)\xi_j \right| \right)\right|\\
& \leq \prod_{\gamma \in \mathcal{M}_{\delta}(\boldsymbol{\xi})} J_0(\varepsilon).
\end{align*}
Using (\ref{eq:bessel3}) of Lemma \ref{lem:besselprop}, the product is bounded by
\begin{align*}
\exp \left( -\frac{\varepsilon^2}{4} \left| \mathcal{M}_{\delta}(\boldsymbol{\xi}) \right|  \right).
\end{align*}
Then, it follows immediately from Lemma \ref{lem:largespec}.
\end{proof}
\end{lemma}

\begin{lemma}\label{lem:approx}
Let  $\delta>0$ be sufficiently small and $\boldsymbol{\xi} \in \mathbb{R}^r.$ Suppose
$\|\boldsymbol{\xi}\| \leq (T\sqrt{r})^{-1}\left(\frac{1}{\delta }\log \frac{1}{\delta} \right)^{1/4}$. Then 
\begin{align*}
\hat{\mu}_{\delta,\boldsymbol{t}}
\left( \frac{\xi_1}{\sqrt{V_1}},\ldots,\frac{\xi_r}{\sqrt{V_r}} \right)=\left( 1+O\left(r^2 T^3\left(\frac{1}{\delta}\log \frac{1}{\delta}\right)^{-1}
\|\boldsymbol{\xi}\|^4\right)\right)\exp \left( -\frac{1}{2}\langle \mathcal{C}\boldsymbol{\xi},\boldsymbol{\xi} \rangle \right).
\end{align*}
\begin{proof}
Applying Lemma \ref{lem:bessel}, we have
\begin{align} \label{eq:beginning}
\log \hat{\mu}_{\delta,\boldsymbol{t}}  \left( \frac{\xi_1}{\sqrt{V_j}},\ldots,\frac{\xi_r}{\sqrt{V_r}} \right)
=\sum_{\gamma>0} \log J_0 \left( 2\left| \sum_{j=1}^r \frac{w_j(\rho)}{\sqrt{V_j}}\xi_j \right| \right).
\end{align}
Applying Proposition \ref{thm:cov} and Lemma \ref{lem:eq:crude} followed by the Cauchy--Schwarz inequality, we obtain
\begin{align*}
2\left| \sum_{j=1}^r \frac{w_j(\rho)}{\sqrt{V_j}}\xi_j \right|
&\ll T \delta \left( \delta \log \frac{1}{\delta} \right)^{-1/2}  \sum_{j=1}^r |\xi_j| \\
&\ll T\left( \frac{1}{\delta}\log \frac{1}{\delta} \right)^{-1/2}r^{1/2}\|\boldsymbol{\xi}\|,
\end{align*}
so by assumption, the left-hand side is at most $1.$ 
Using (\ref{eq:bessel4}) of Lemma \ref{lem:besselprop}, 
the expression (\ref{eq:beginning}) becomes
\begin{align*}
-\sum_{\gamma>0}\left| \sum_{j=1}^r \frac{w_j(\rho)}{\sqrt{V_j}}\xi_j \right|^2+O \left( \sum_{\gamma>0} \left| \sum_{j=1}^r \frac{w_j(\rho)}{\sqrt{V_j}}\xi_j \right|^4 \right).
\end{align*}
Opening the square, the main term becomes
\begin{align*}
-\frac{1}{2}\sum_{1 \leq j,k \leq r}\left(\Re \left(2\sum_{\gamma>0} \frac{w_j(\rho)\overline{w_k(\rho)}}{\sqrt{V_j}\sqrt{V_k}}\right)\right)\xi_j\xi_k
&=-\frac{1}{2}\sum_{1 \leq j,k \leq r}\left(\sum_{\gamma} \frac{w_j(\rho)\overline{w_k(\rho)}}{\sqrt{V_j}\sqrt{V_k}}\right)\xi_j\xi_k \\
&=-\frac{1}{2}\langle \mathcal{C}\boldsymbol{\xi},\boldsymbol{\xi} \rangle.
\end{align*}

On the other hand, we write
\begin{align*}
\sum_{\gamma>0}\left| \sum_{j=1}^r \frac{w_j(\rho)}{\sqrt{V_j}}\xi_j \right|^4 = \sum_{0<\gamma \leq \frac{1}{T\delta}}\left| \sum_{j=1}^r \frac{w_j(\rho)}{\sqrt{V_j}}\xi_j \right|^4
+\sum_{\gamma>\frac{1}{T\delta}}\left| \sum_{j=1}^r \frac{w_j(\rho)}{\sqrt{V_j}}\xi_j \right|^4.
\end{align*}
Applying Proposition \ref{thm:cov} and Lemma \ref{lem:eq:crude}, the first sum is
\begin{align*}
\ll N \left(\frac{1}{T\delta}\right) \frac{(T\delta)^4}{\left( \delta \log \frac{1}{\delta} \right)^2}
 \left(\sum_{j=1}^r |\xi_j| \right)^4 
\ll r^2 T^3\left(\frac{1}{\delta}\log \frac{1}{\delta}\right)^{-1}\|\boldsymbol{\xi}\|^4
\end{align*}
by the Cauchy--Schwarz inequality. Similarly, the second sum is
\begin{align*}
\ll \left(\sum_{\gamma>\frac{1}{T\delta}} \frac{1}{|\gamma|^4} \right) \frac{r^2}{\left( \delta \log \frac{1}{\delta} \right)^2}\|\boldsymbol{\xi}\|^4
\ll r^2 T^{3} \left(\frac{1}{\delta}\log \frac{1}{\delta}\right)^{-1} \|\boldsymbol{\xi}\|^4.
\end{align*}
Therefore, we have
\begin{align*}
\log \hat{\mu}_{\delta,\boldsymbol{t}}  \left( \frac{\xi_1}{\sqrt{V_j}},\ldots,\frac{\xi_r}{\sqrt{V_r}} \right)
=-\frac{1}{2}\langle \mathcal{C}\boldsymbol{\xi},\boldsymbol{\xi}\rangle
+O\left( r^2 T^{3} \left(\frac{1}{\delta}\log \frac{1}{\delta}\right)^{-1} \|\boldsymbol{\xi}\|^4 \right).
\end{align*}
Then, the lemma follows from exponentiating both sides.
\end{proof}
\end{lemma}

To prove the main theorem, we require two additional lemmas.

\begin{lemma} \label{lem:cxx}
Let $r \geq 1$ and $\boldsymbol{x} \in \mathbb{R}^r.$ Then
\begin{align*}
\langle \mathcal{C}\boldsymbol{x}, \boldsymbol{x}\rangle =\left( 1+O\left( \frac{\log 2r}{\log \frac{1}{\delta}} \right) \right)\|\boldsymbol{x}\|^2.
\end{align*}

\begin{proof}
We have
\begin{align*}
\langle \mathcal{C} \boldsymbol{x}, \boldsymbol{x}\rangle&=\sum_{1 \leq j,k \leq r}c_{jk}x_jx_k\\
&=\|\boldsymbol{x}\|^2+\sum_{1 \leq j \neq k \leq r}c_{jk}x_jx_k.
\end{align*}
Using the AM-GM inequality, the sum is
\begin{align*}
\ll& \sum_{1 \leq j \leq r} |x_j|^2 \sum_{\substack{1 \leq k \leq r\\k \neq j}} |c_{jk}|+
\sum_{1 \leq k \leq r} |x_k|^2 \sum_{\substack{1 \leq j \leq r\\j \neq k}} |c_{jk}|.
\end{align*}
By Proposition \ref{thm:cov} with Lemma \ref{lem:eq:coulomb}, and the assumption $|t_j-t_k|\geq 1$ for $1 \leq j \neq k \leq r,$ this is
\begin{align*}
\ll \frac{1}{\log \frac{1}{\delta}}
\sum_{1 \leq j \leq r} |x_j|^2 \sum_{\substack{1 \leq k \leq r\\k \neq j}} \frac{1}{|t_j-t_k|}+ \frac{1}{\log \frac{1}{\delta}}
\sum_{1 \leq k \leq r} |x_k|^2 \sum_{\substack{1 \leq j \leq r\\j \neq k}} \frac{1}{|t_j-t_k|} 
\ll \frac{\log 2r}{\log \frac{1}{\delta}}\|\boldsymbol{x}\|^2,
\end{align*}
and hence the lemma follows.
\end{proof}

\end{lemma}

\begin{lemma}\label{lem:comp}
Let $r \geq 1$ and $\boldsymbol{x} \in \mathbb{R}^r.$ Suppose  $R >  \sqrt{2r\log 2r}.$ Then
\begin{gather*}
\frac{1}{(2\pi)^r} \int_{\|\boldsymbol{\xi}\| \leq R}
e^{i\langle \boldsymbol{x},\boldsymbol{\xi} \rangle} \exp \left( -\frac{1}{2}
 \langle \mathcal{C}\boldsymbol{\xi},\boldsymbol{\xi}\rangle \right) d\boldsymbol{\xi} \\
=\frac{1}{(2\pi)^{r/2}(\det \mathcal{C})^{1/2}}\exp \left(-\frac{1}{2} \langle \mathcal{C}^{-1}\boldsymbol{x},\boldsymbol{x} \rangle \right)
+O\left( \exp\left(-\frac{R^2}{4} \right) \right).
\end{gather*}
\begin{proof}
Applying the Fourier inversion formula, we have
\begin{align*}
\frac{1}{(2\pi)^r} \int_{\boldsymbol{\xi} \in \mathbb{R}^r}
e^{i\langle \boldsymbol{x},\boldsymbol{\xi} \rangle} \exp \left( -\frac{1}{2}
 \langle \mathcal{C}\boldsymbol{\xi},\boldsymbol{\xi}\rangle \right) d\boldsymbol{\xi}
=\frac{1}{(2\pi)^{r/2}(\det \mathcal{C})^{1/2}}&\exp \left(-\frac{1}{2} \langle \mathcal{C}^{-1}\boldsymbol{x},\boldsymbol{x} \rangle \right).
\end{align*}
On the other hand, it follows from Lemma \ref{lem:cxx} that
\begin{align*}
\frac{1}{(2\pi)^r} \int_{\|\boldsymbol{\xi}\| > R}
 \exp \left( -\frac{1}{2}
 \langle \mathcal{C}\boldsymbol{\xi},\boldsymbol{\xi}\rangle \right) d\boldsymbol{\xi}
 &= \frac{1}{(2\pi)^r}
 \int_{\| \boldsymbol{\xi} \|>R}
 \exp \left( -\frac{1}{2} \left( 1+O\left( \frac{\log 2r}{\log \frac{1}{\delta}} \right) \right)\|\boldsymbol{\xi}\|^2 \right) d\boldsymbol{\xi} \\
 &\leq  \frac{1}{(2\pi)^r}
 \int_{\| \boldsymbol{\xi} \|>R}
 \exp \left( -\frac{1}{3} \|\boldsymbol{\xi}\|^2 \right) d\boldsymbol{\xi}.
\end{align*}
Using spherical coordinates, one can show that this is
\begin{align*}
\frac{1}{(2\pi)^r} \cdot \frac{2\pi^{r/2}}{\Gamma(r/2)} 
\int_R^{\infty}
\exp \left( -\frac{1}{3}x^2 \right) x^{r-1} dx
\ll \exp\left(-\frac{R^2}{4} \right),
\end{align*}
and hence the lemma follows.
\end{proof}
\end{lemma}

\begin{proof}[Proof of Theorem \ref{thm:clt}]
The proof consists of four steps: Fourier inversion, localization, approximation, and completion.
By definition, we have
\begin{align*}
\mathbb{P}_x^{\log}( 
\widetilde{\boldsymbol{E}}(x;\delta,\boldsymbol{t})
\in B )
=\int_{\boldsymbol{x}\in \mathbb{R}^r \,:\,\left( \frac{x_1}{\sqrt{V_1}},\ldots, \frac{x_r}{\sqrt{V_r}} \right) \in B}
d\mu_{\delta;\boldsymbol{t}}(\boldsymbol{x}),
\end{align*}
which is by the Fourier inversion formula
\begin{align*}
\frac{1}{(2\pi)^r}\int_{\boldsymbol{x}\in \mathbb{R}^r \,:\,\left( \frac{x_1}{\sqrt{V_1}},\ldots, \frac{x_r}{\sqrt{V_r}} \right) \in B}\int_{\boldsymbol{\xi} \in \mathbb{R}^r}e^{i\langle \boldsymbol{x},\boldsymbol{\xi}  \rangle} \hat{\mu}_{\delta,\boldsymbol{t}}(\boldsymbol{\xi}) d\boldsymbol{\xi} d\boldsymbol{x}.
\end{align*}
Making the change of variables $\boldsymbol{\omega}=\left(\frac{1}{\delta}\log \frac{1}{\delta}\right)^{-1}\boldsymbol{\xi}$ and 
$\boldsymbol{y}=\left(\frac{1}{\delta} \log \frac{1}{\delta} \right)\boldsymbol{x},$ 
this becomes
\begin{align} \label{eq:afterfourier}
\frac{1}{(2\pi)^r}\int_{\boldsymbol{y} \in \mathbb{R}^r\,:\,\left( \frac{ y_1}{\sqrt{V_1}},\ldots, \frac{ y_r}{\sqrt{V_r}} \right) \in \left( \frac{1}{\delta} \log \frac{1}{\delta} \right)B}
\int_{\boldsymbol{\omega}\in \mathbb{R}^r}
e^{i\langle \boldsymbol{y},\boldsymbol{\omega}\rangle}
\hat{\mu}_{\delta,\boldsymbol{t}}
\left(\left(\frac{1}{\delta} \log \frac{1}{\delta}\right)\boldsymbol{\omega}\right) d\boldsymbol{\omega} d\boldsymbol{y}.
\end{align}
As we will see, the main contribution comes from $\|\boldsymbol{\omega}\| \leq \varepsilon,$ where $0<\varepsilon \leq 1$ will be determined later. For $\|\boldsymbol{\omega}\| > \varepsilon,$ we write
\begin{gather*}
\int_{\|\boldsymbol{\omega}\|>\varepsilon}
e^{i\langle \boldsymbol{y},\boldsymbol{\omega}\rangle}
\hat{\mu}_{\delta,\boldsymbol{t}}\left(\left(\frac{1}{\delta} \log \frac{1}{\delta}\right)\boldsymbol{\omega}\right) d\boldsymbol{\omega} \\
=\left\{\int_{\varepsilon<\|\boldsymbol{\omega}\| \leq 2}+
\int_{\|\boldsymbol{\omega}\| > 2}\right\}
e^{i\langle \boldsymbol{y},\boldsymbol{\omega}\rangle}
\hat{\mu}_{\delta,\boldsymbol{t}}\left(\left(\frac{1}{\delta} \log \frac{1}{\delta}\right)\boldsymbol{\omega}\right) d\boldsymbol{\omega}.
\end{gather*}
Applying Lemma \ref{lem:loc}, the first integral is
\begin{align}\label{eq:first}
\ll \frac{\pi^{r/2}2^r}{\Gamma\left( \frac{r}{2}+1 \right)}\exp \left( -\dfrac{\varepsilon^2}{1000rT^2}\left(\dfrac{1}{\delta}\log \dfrac{1}{\delta}\right)\right) \ll
\exp \left( -\dfrac{\varepsilon^2}{1000rT^2}\left(\dfrac{1}{\delta}\log \dfrac{1}{\delta}\right)\right).
\end{align}
Meanwhile, the second integral is bounded by
\begin{gather*} 
\sum_{j=1}^{\infty} \int_{2^j<\|\boldsymbol{\omega}\|\leq 2^{j+1}}
\left| \hat{\mu}_{\delta,\boldsymbol{t}}\left(\left(\frac{1}{\delta} \log \frac{1}{\delta}\right)\boldsymbol{\omega}\right) \right| d\boldsymbol{\omega} \nonumber \\
\leq \sum_{j=1}^{\infty} \meas \left( \left\{ 2^j<\| \boldsymbol{\omega} \| \leq 2^{j+1} \right\} \right)2^{-\frac{j}{500rT^2}\left(\frac{1}{\delta}\log \frac{1}{\delta}\right)},
\end{gather*}
which is
\begin{gather}
 \ll \sum_{j=1}^{\infty}2^{(j+1)r-\frac{j}{500rT^2}\left(\frac{1}{\delta}\log \frac{1}{\delta}\right)}
\ll \exp \left( -\frac{1}{1000rT^2}\left(\frac{1}{\delta}\log \frac{1}{\delta}\right)\right). \label{eq:second}
\end{gather}
Combining (\ref{eq:first}) and (\ref{eq:second}), the contribution of $\|\boldsymbol{\omega}\|>\varepsilon$ to (\ref{eq:afterfourier}) is
\begin{gather*}
\frac{1}{(2\pi)^r}\int_{\boldsymbol{y} \in \mathbb{R}^r\,:\,\left( \frac{ y_1}{\sqrt{V_1}},\ldots, \frac{ y_r}{\sqrt{V_r}} \right) \in \left( \frac{1}{\delta}\log \frac{1}{\delta} \right)B}
\int_{\|\boldsymbol{\omega}\|>\varepsilon}
e^{i\langle \boldsymbol{y},\boldsymbol{\omega}\rangle}
\hat{\mu}_{\delta,\boldsymbol{t}}
\left(\left( \frac{1}{\delta}\log \frac{1}{\delta} \right)\boldsymbol{\omega}\right) d\boldsymbol{\omega} d\boldsymbol{y}\\
\ll  \frac{1}{\pi^r} \cdot \meas(B)\left( \frac{1}{\delta}\log^3 \frac{1}{\delta} \right)^{r/2} \exp \left( -\frac{\varepsilon^2}{1000rT^2}\left(\frac{1}{\delta}\log \frac{1}{\delta}\right) \right)  .
\end{gather*}
Taking
$\varepsilon=1000rT\sqrt{\delta},$ which is at most $1$ by the assumptions on $r, T,$ 
this is
\begin{align} \label{eq:et1}
\ll \frac{1}{\pi^r} \cdot \meas(B)\exp \left( -100r \log \frac{1}{\delta} \right),
\end{align}
which is negligible. Therefore, we are left with
\begin{align*}
\frac{1}{(2\pi)^r}\int_{\boldsymbol{y} \in \mathbb{R}^r\,:\,\left( \frac{ y_1}{\sqrt{V_1}},\ldots, \frac{ y_r}{\sqrt{V_r}} \right) \in 
\left( \frac{1}{\delta}\log \frac{1}{\delta} \right)B}
\int_{\|\boldsymbol{\omega}\| \leq \varepsilon}
e^{i\langle \boldsymbol{y},\boldsymbol{\omega}\rangle}
\hat{\mu}_{\delta,\boldsymbol{t}}
\left(\left( \frac{1}{\delta}\log \frac{1}{\delta} \right)\boldsymbol{\omega}\right) d\boldsymbol{\omega} d\boldsymbol{y},
\end{align*}
which is by Proposition \ref{thm:cov}
\begin{gather*}
\frac{1}{(2\pi)^r}\int_{\boldsymbol{y} \in \mathbb{R}^r\,:\,\left( \frac{ y_1}{\sqrt{V_1}},\ldots, \frac{ y_r}{\sqrt{V_r}} \right) \in 
\left( \frac{1}{\delta}\log \frac{1}{\delta} \right)B}
 \\ \int_{\left\|(\omega_1 \sqrt{V_1},\ldots,\omega_r \sqrt{V_r})\right\| \leq  \left(\delta\log \frac{1}{\delta}\right)^{1/2}\varepsilon}
e^{i\langle \boldsymbol{y},\boldsymbol{\omega}\rangle}
\hat{\mu}_{\delta,\boldsymbol{t}}
\left(\left( \frac{1}{\delta}\log \frac{1}{\delta} \right)\boldsymbol{\omega}\right) d\boldsymbol{\omega} d\boldsymbol{y}\\
+O\left(\frac{1}{(2\pi)^r}\int_{\boldsymbol{y} \in \mathbb{R}^r\,:\,\left( \frac{ y_1}{\sqrt{V_1}},\ldots, \frac{ y_r}{\sqrt{V_r}} \right) \in 
\left( \frac{1}{\delta}\log \frac{1}{\delta} \right)B}
\int_{\|\boldsymbol{\omega}\| > \frac{\varepsilon}{2}}
\left|\hat{\mu}_{\delta,\boldsymbol{t}}
\left(\left( \frac{1}{\delta}\log \frac{1}{\delta} \right)\boldsymbol{\omega}\right)\right| d\boldsymbol{\omega} d\boldsymbol{y}\right).
\end{gather*}
Here, the error term is again
\begin{align} \label{eq:etencore}
\ll \frac{1}{\pi^r} \meas(B)\exp \left( -100r \log \frac{1}{\delta} \right).
\end{align}
Making the change of variables $\xi_j=\sqrt{V_j}\left(\frac{1}{\delta}\log \frac{1}{\delta}\right)\omega_j$ and 
$x_j=\left( \sqrt{V_j}\left(\frac{1}{\delta}\log \frac{1}{\delta}\right) \right)^{-1}y_j$ for $j=1,\ldots,r,$ the main term becomes
\begin{align*} 
\frac{1}{(2\pi)^r}\int_{\boldsymbol{x} \in B}
\int_{\|\boldsymbol{\xi}\| \leq 1000rT\log^{3/2}\frac{1}{\delta}}
e^{i\langle \boldsymbol{x},\boldsymbol{\xi}\rangle}
\hat{\mu}_{\delta,\boldsymbol{t}}
\left(\frac{\xi_1}{\sqrt{V_1}},\ldots,\frac{\xi_r}{\sqrt{V_r}}\right) d\boldsymbol{\xi} d\boldsymbol{x}.
\end{align*}
Since $T \leq \delta^{-\frac{1}{10}}$ and $r \leq \log \frac{1}{\delta},$ it follows from Lemma \ref{lem:approx} that this is
\begin{gather} 
\frac{1}{(2\pi)^r}\int_{\boldsymbol{x} \in B}
\int_{\|\boldsymbol{\xi}\| \leq 1000rT\log^{3/2}\frac{1}{\delta}}
e^{i\langle \boldsymbol{x},\boldsymbol{\xi}\rangle}
\left( 1+O\left(r^2 T^3\left(\frac{1}{\delta}\log \frac{1}{\delta}\right)^{-1}
\|\boldsymbol{\xi}\|^4\right)\right) \nonumber\\
\cdot \exp \left( -\frac{1}{2}\langle \mathcal{C}\boldsymbol{\xi},\boldsymbol{\xi} \rangle \right) d\boldsymbol{\xi} d\boldsymbol{x}. \label{eq:aftertaylor}
\end{gather}
The contribution of the error term above is
\begin{align} \label{eq:et3}
&\ll r^2 T^3  \left(\frac{1}{\delta}\log \frac{1}{\delta}\right)^{-1} \meas(B) \cdot \frac{1}{(2\pi)^r}\int_{\mathbb{R}^r} 
\|\boldsymbol{\xi}\|^4 
\exp\left( -\frac{1}{2} \langle \mathcal{C}\boldsymbol{\xi},\boldsymbol{\xi} \rangle \right) d\boldsymbol{\xi}.
\end{align}
Applying Lemma \ref{lem:cxx}, we need to evaluate 
the integral
\begin{align*}
\frac{1}{(2\pi)^r}\int_{\mathbb{R}^r} 
\|\boldsymbol{\xi}\|^4 
\exp\left( -\frac{1}{2} \left(1+O\left( \frac{\log 2r}{\log 1/\delta} \right) \right) \|\boldsymbol{\xi}\|^2  \right) d\boldsymbol{\xi},
\end{align*}
which is after a suitable change of variables
\begin{align*}
\ll 
\left(1+O\left( \frac{\log 2r}{\log 1/\delta} \right) \right)^r  \frac{1}{(2\pi)^{r}}
\int_{\mathbb{R}^r} 
\|\boldsymbol{\xi}\|^4 
\exp\left( -\frac{1}{2}\|\boldsymbol{\xi}\|^2  \right) d\boldsymbol{\xi}.
\end{align*}
Using spherical coordinates, one can show that 
\begin{align*}
\frac{1}{(2\pi)^{r}}
\int_{\mathbb{R}^r} 
\|\boldsymbol{\xi}\|^4 
\exp\left( -\frac{1}{2}\|\boldsymbol{\xi}\|^2  \right) d\boldsymbol{\xi}=
\frac{1}{(2\pi)^r} \cdot \frac{2\pi^{r/2}}{\Gamma(r/2)}
\int_{0}^{\infty}x^4\exp\left(-\frac{1}{2}x^2 \right)
x^{r-1}dx.
\end{align*}
Making the change of variables $y=\frac{1}{2}x^2,$ this becomes
\begin{align*} 
\frac{4}{(2\pi)^{r/2}} \cdot \frac{1}{\Gamma(r/2)}
\int_{0}^{\infty}y^{\frac{r}{2}+2}e^{-y} \frac{dy}{y}&=
\frac{4}{(2\pi)^{r/2}} \cdot \frac{\Gamma(r/2+2)}{\Gamma(r/2)}\\
&=\frac{4}{(2\pi)^{r/2}}\left(\frac{r}{2}+1\right)\frac{r}{2}.
\end{align*}
Therefore, since $ r/\log \frac{1}{\delta}$ is sufficiently small, the expression (\ref{eq:et3}) is
\begin{align} \label{eq:etalmost}
\ll  r^4T^3 \left( \frac{1}{\sqrt{2\pi}}
+O \left( \frac{\log 2r}{\log 1/\delta} \right) \right)^r \left(\frac{1}{\delta}
\log \frac{1}{\delta} \right)^{-1} \meas(B).
\end{align}

On the other hand, by Lemma \ref{lem:comp}, the main term of (\ref{eq:aftertaylor}) is
\begin{align} \label{eq:et4}
\frac{1}{(2\pi)^{r/2}(\det \mathcal{C})^{1/2}}\int_{\boldsymbol{x} \in B}\exp \left(-\frac{1}{2} \langle \mathcal{C}^{-1}\boldsymbol{x},\boldsymbol{x} \rangle \right)  d\boldsymbol{x}
+O\left( \exp\left(-r^2T^2\log^3 \frac{1}{\delta} \right) \right).
\end{align}
%
Finally, the theorem follows from collecting the error terms in (\ref{eq:et1}), (\ref{eq:etencore}), (\ref{eq:etalmost}) and (\ref{eq:et4}).
\end{proof}

\section{Proof of Theorem \ref{cor:compare}}

Throughout the section, we always assume 
$\delta>0$ is sufficiently small, $1 \leq r \leq \frac{\log 1/\delta}{\log \log 1/\delta},$
$T=T_{[r]} \leq \delta^{-\frac{1}{10}},$ and $|t_j-t_k|\geq 1$ for $1 \leq j \neq k \leq r$ as in the statement of Theorem \ref{cor:compare}. We begin with a lemma concerning the determinant and inverse of almost identity matrices.
\begin{lemma} \label{lem:matrix}
Define $M_r(\varepsilon)$ as the set of all $r \times r$
symmetric matrices whose diagonal entries are $1$, and whose off-diagonal entries have
absolute value at most $\varepsilon \leq 1/2r.$
Let $A = (a_{jk} )$ be a matrix in $M_r(\varepsilon).$ Then
\begin{align}\label{eq:det}
\det(A)=1+O\left(  \sum_{1\leq j \neq k \leq r} |a_{jk}|^2
\right).
\end{align}
Moreover, the matrix $A$ is invertible and if $\tilde{a}_{jk}$'s denote the entries of
the inverse matrix $A^{-1},$ then 
\begin{align*}
\tilde{a}_{jk}=
\begin{cases}
\hfil 1+O( \sum_{1\leq l \neq m \leq r} |a_{lm}|^2 )&\mbox{{\normalfont if $j=k,$} }\\
-a_{jk}+O\left(\sum_{\substack{1 \leq i \leq r\\i \neq j,k}} |a_{ji}a_{ik}|+\sum_{\substack{1 \leq g \neq h \leq r\\g,h \neq j,k}} |a_{jg}a_{gh}a_{hk}|+\varepsilon
\sum_{1\leq l \neq m \leq r}|a_{lm}|^2
\right)&\mbox{{\normalfont if $j \neq k.$} }
\end{cases}
\end{align*}
\end{lemma}
To apply the lemma, let $1 \leq r \leq \frac{1}{2}\log \frac{1}{\delta}$ and $\varepsilon=\left(\log \frac{1}{\delta}\right)^{-1}.$ Then, by Proposition \ref{thm:cov} with Lemma \ref{lem:eq:coulomb}, the matrix
$\mathcal{C}$ is in $M_r(\varepsilon)$ as the entries
\begin{align*}
|c_{jk}|=\frac{\left|\Cov_{jk}\right| }{\sqrt{V_j V_k}}\leq  \left(\log \frac{1}{\delta}\right)^{-1} 
\end{align*}
for $1\leq j \neq k \leq r.$
\begin{proof}
This is a sharper version of \cite[Lemma 3.3]{MR3773805}.
 Let us estimate the determinant first. We have
\begin{align} \label{eq:detformula}
\det(A)=1+\sum_{\substack{\sigma \in S_r\\ \sigma \neq e}}
\sgn(\sigma) a_{1,\sigma(1)} \cdots a_{r,\sigma(r)},
\end{align}
where $S_r$ denotes the symmetric group on $r$ elements. We divide the sum according to the number of points $t$ that are not fixed by $\sigma.$ Note that the only term
with $t = 0$ is the identity permutation $e$, which has been isolated from the sum. There are no
terms with $t = 1,$ and the contribution from $t = 2$ is at most
\begin{align} \label{eq:t=2}
\sum_{1 \leq j < k \leq r} |a_{jk}|^2.
\end{align}
For each $3 \leq t \leq r$, by averaging the total contribution is at most
\begin{gather*}
\frac{1}{t} \sum_{1 \leq j \leq r} \sum_{\substack{\sigma \in S_r, \sigma(j) \neq j \\ \text{$\sigma$ has $t$ non-fixed points}}}
|a_{j,\sigma(j)}||a_{\sigma^{-1}(j),j}|\varepsilon^{t-2} \\
= 
\frac{1}{t}\sum_{\substack{1 \leq j,k,l \leq r\\j \neq k,l}}|a_{jk}a_{lj}|
\varepsilon^{t-2} \sum_{\substack{\sigma \in S_r, \sigma(j) = k, \sigma^{-1}(j) = l  \\ \text{$\sigma$ has $t$ non-fixed points}}} 1.
\end{gather*}
We split the sum into the diagonal sum 
\begin{align*}
\Sigma_{k=l}:=\frac{1}{t}\sum_{\substack{1 \leq j,k \leq r\\j \neq k}}|a_{jk}|^2
\varepsilon^{t-2} \sum_{\substack{\sigma \in S_r, \sigma(j) = k, \sigma^{-1}(j) = k  \\ \text{$\sigma$ has $t$ non-fixed points}}} 1
\end{align*}
and the off-diagonal sum 
\begin{align*}
\Sigma_{k \neq l}:=\frac{1}{t}\sum_{\substack{1 \leq j,k,l \leq r\\j, k,l \text{ distinct}}}|a_{jk}a_{lj}|
\varepsilon^{t-2} \sum_{\substack{\sigma \in S_r, \sigma(j) = k, \sigma^{-1}(j) = l  \\ \text{$\sigma$ has $t$ non-fixed points}}} 1.
\end{align*}
Then by simple combinatorics, we have
\begin{align} \label{eq:k=l}
\Sigma_{k=l}
= & \frac{1}{t}\sum_{\substack{1 \leq j,k \leq r\\j \neq k}}|a_{jk}|^2
\varepsilon^{t-2} {r-2 \choose t-2} (t-2)! \nonumber \\
\leq & (\varepsilon r)^{t-2} \sum_{\substack{1 \leq j,k \leq r\\j \neq k}}|a_{jk}|^2.
\end{align}
Similarly, we have
\begin{align*}
\Sigma_{k \neq l} 
= & \frac{1}{t}\sum_{\substack{1 \leq j,k,l \leq r\\j, k,l \text{ distinct}}}|a_{jk}a_{lj}|
\varepsilon^{t-2} { r-3 \choose t-3 } (t-2)!\\
\leq & \varepsilon^{t-2}r^{t-3} 
\sum_{\substack{1 \leq j,k,l \leq r\\j, k,l \text{ distinct}}}|a_{jk}a_{lj}|,
\end{align*}
which is by the AM-GM inequality
\begin{gather} 
\leq  \varepsilon^{t-2}r^{t-3} 
\sum_{\substack{1 \leq j,k,l \leq r\\j, k,l \text{ distinct}}} \frac{1}{2}(|a_{jk}|^2+|a_{lj}|^2) \nonumber\\
\leq  (\varepsilon r)^{t-2} 
\sum_{\substack{1 \leq j,k \leq r\\j \neq k}} |a_{jk}|^2. \label{eq:kneql}
\end{gather}
Combining (\ref{eq:detformula}), (\ref{eq:t=2}), (\ref{eq:k=l}) and (\ref{eq:kneql}), the expression (\ref{eq:det}) follows from summing over $3 \leq t \leq r.$ Also, since by assumption $|a_{jk}| \leq \varepsilon \leq 1/2r$, the determinant satisfies
\begin{align*}
\det(A) \geq  
1- \frac{1}{2}\sum_{\substack{1 \leq j,k \leq r\\j \neq k}} |a_{jk}|^2 - 2\sum_{3 \leq t \leq r} (\varepsilon r)^{t-2} \sum_{\substack{1 \leq j,k \leq r\\j \neq k}} |a_{jk}|^2 \geq \frac{3}{8}.
\end{align*}
In particular, the matrix $A$ is invertible. 

It remains to estimate the entries of the inverse matrix $A^{-1}.$
Recall that
\begin{align} \label{eq:linearalbegra}
\tilde{a}_{jk}=(-1)^{j+k}\frac{\det A_{kj}}{\det A},
\end{align}
where $A_{kj}$ denotes the matrix $A$ with the $k$-th row and the $j$-th column removed. If $j=k,$ then 
\begin{align*}
\tilde{a}_{jj}=\frac{\det A_{jj}}{\det A}.
\end{align*}
Since $A_{jj} \in M_{r-1}(\varepsilon),$ it follows from (\ref{eq:det}) that 
\begin{align*}
\tilde{a}_{jj}=1+O\left( \sum_{1\leq l \neq m \leq r} |a_{lm}|^2 \right).
\end{align*}

Without loss of generality, suppose $j< k.$ Let $\mathcal{B}_{kj}$ denote the set of bijections from $[r]\setminus\{k\}$ to $[r]\setminus \{j\}.$ Then 
\begin{align} \label{eq:altersum}
\det A_{kj}=\sum_{\sigma \in \mathcal{B}_{kj}}
\sgn(\sigma)\prod_{i \neq k}a_{i,\sigma(i)},
\end{align}
where $\sgn(\sigma):=\sgn(\tilde{\sigma})$ with $\tilde{\sigma} \in S_{[r]\setminus \{k\}}$ denoting the permutation
\begin{align*}
\tilde{\sigma}(i):=
\begin{cases}
\sigma(i)-1 & \mbox{{\normalfont if $j<\sigma(i)\leq k,$ } }\\
\hfil \sigma(i) & \mbox{{\normalfont otherwise } }
\end{cases}
\end{align*}
for $\sigma \in \mathcal{B}_{kj}.$ Again, we divide the sum (\ref{eq:altersum}) according to the number of points $t$ that are not fixed by $\sigma.$ Note that there are no terms with $t=0,$ and the only term with $t=1$ is $a_{jk},$ whose contribution to (\ref{eq:linearalbegra}) is 
\begin{align} \label{eq:t=1new}
(-1)^{j+k}\frac{(-1)^{k-j+1}a_{jk}}{\det A}
=-a_{jk}+O\left( \varepsilon 
\sum_{1\leq l \neq m \leq r}|a_{lm}|^2
\right)
\end{align}
using (\ref{eq:det}). Also, the contribution from $t=2$ is at most
\begin{align} \label{eq:t=2new}
\sum_{\substack{1 \leq i \leq r\\i \neq j,k}} |a_{ji}a_{ik}|,
\end{align}
and the contribution from $t=3$ is at most
\begin{gather}
|a_{jk}|\sum_{\substack{1 \leq g \neq h \leq r}} |a_{gh}a_{hg}|
+\sum_{\substack{1 \leq g \neq h \leq r\\g,h \neq j,k}} |a_{jg}a_{gh}a_{hk}| \nonumber \\
\leq \varepsilon \sum_{\substack{1 \leq g \neq h \leq r}} |a_{gh}|^2
+\sum_{\substack{1 \leq g \neq h \leq r\\g,h \neq j,k}} |a_{jg}a_{gh}a_{hk}|. \label{eq:t=3new}
\end{gather}
Finally, for any $4 \leq t \leq r-1,$ each $\sigma \in \mathcal{B}_{kj}$ must have at least $t-3$ non-fixed points distinct from $j,$ and whose image and pre-image are distinct from $k$ and $j,$ respectively. By averaging, the total contribution is at most     
\begin{gather*}
\frac{1}{t-3} \sum_{\substack{1 \leq i \leq r\\i \neq j,k}} \sum_{\substack{\sigma \in \mathcal{B}_{kj}; \sigma(i), \sigma^{-1}(i) \neq i,j,k\\ \text{$\sigma$ has $t$ non-fixed points}}}  | a_{i,\sigma(i)}| | a_{\sigma^{-1}(i),i}  | \varepsilon^{t-2} \\
= \frac{1}{t-3} \sum_{\substack{i \neq j,k\\g \neq i,j,k\\h \neq i,j,k}}
|a_{ig}a_{hi}| \varepsilon^{t-2} \sum_{\substack{ \sigma \in \mathcal{B}_{kj}; \sigma(i)=g, \sigma^{-1}(i)=h  \\ \text{$\sigma$ has $t$ non-fixed points}} } 1.
\end{gather*}
Splitting into the diagonal and the off-diagonal sum, this is
\begin{gather}
= \frac{1}{t-3} \sum_{\substack{i \neq j,k\\g \neq i,j,k}}
|a_{ig}|^2 \epsilon^{t-2} {r-4 \choose t-3} (t-2)!
+ \frac{1}{t-3} \sum_{\substack{i \neq j,k\\g \neq i,j,k\\h \neq i,j,k}}
|a_{ig}a_{hi}| \varepsilon^{t-2} {r-5 \choose t-4} (t-2)! \nonumber\\
\ll  \varepsilon  (\varepsilon r)^{t-3} t \sum_{\substack{1 \leq i,g \leq r\\i \neq g}} |a_{ig}|^2. \label{eq:tgeq3new}
\end{gather}
Combining (\ref{eq:t=1new}), (\ref{eq:t=2new}), (\ref{eq:t=3new}) and (\ref{eq:tgeq3new}), the lemma follows from summing over $4 \leq t \leq r-1.$
\end{proof}

We can now estimate the probability density function of the $r$-dimensional normal distribution with mean $\boldsymbol{0}$ and covariance matrix $\mathcal{C}.$
\begin{lemma} \label{lem:finale}
Let $1 \leq r \leq \frac{1}{2}\log \frac{1}{\delta}.$ Then
\begin{gather*}
\frac{1}{(2\pi)^{r/2}(\det\mathcal{C})^{1/2}} 
\exp \left( -\frac{1}{2}\langle \mathcal{C}^{-1}\boldsymbol{x},\boldsymbol{x}\rangle  \right)\\
=
\left( 1+O\left(\frac{r}{\log^2 \frac{1}{\delta}}\right) \right)
 \frac{1}{(2\pi)^{r/2}}
\exp\left(-\frac{1}{2}\|\boldsymbol{x}\|^2
\left( 1+O\left(\frac{\log 2r}{\log \frac{1}{\delta}} \right) \right)\right).
\end{gather*}
More precisely, this is
\begin{align*}
\left( 1+O\left(\frac{r}{\log^2 \frac{1}{\delta}}\right) \right)\frac{1}{(2\pi)^{r/2}}
\exp\left(-\frac{1}{2}\left( 1+O\left( 
\frac{r}{\log^2 \frac{1}{\delta}} \right) \right)\|\boldsymbol{x}\|^2+\sum_{1 \leq j<k\leq r}c_{jk}x_jx_k \right).
\end{align*}
\begin{proof}
By Proposition \ref{thm:cov} with Lemma \ref{lem:eq:coulomb}, the assumption $|t_j-t_k|\geq 1$ for $1 \leq j \neq k \leq r$ implies that
\begin{align*} 
\sum_{1 \leq j\neq k \leq r}|c_{jk}|^2
&\ll \frac{1}{\log^2 \frac{1}{\delta}}\sum_{1 \leq j \leq r}
\sum_{\substack{1 \leq k \leq r\\k \neq j}}\frac{1}{|t_j-t_k|^2} \\
&\ll \frac{r}{\log^2 \frac{1}{\delta}}.
\end{align*}
Applying Lemma \ref{lem:matrix}, we have
\begin{align*}
\det(\mathcal{C})=1+O\left( \frac{r}{\log^2 \frac{1}{\delta}} \right),
\end{align*}
and hence
\begin{align} \label{eq:reciprocal}
\frac{1}{\det(\mathcal{C})^{1/2}}=1+O\left( \frac{r}{\log^2 \frac{1}{\delta}} \right).
\end{align}

On the other hand, again by Lemma \ref{lem:matrix}, we have
\begin{align*}
\langle \mathcal{C}^{-1} \boldsymbol{x}, \boldsymbol{x} \rangle
&=\sum_{1 \leq j,k \leq r} \tilde{c}_{jk}x_jx_k\\
&=\left( 1+O\left(\frac{r}{\log^2 \frac{1}{\delta}}\right)\right)\|\boldsymbol{x}\|^2+
\sum_{1 \leq j \neq k \leq r}\tilde{c}_{jk}x_jx_k,
\end{align*}
where the last sum is
\begin{gather*}
-\sum_{1 \leq j\neq k \leq r} c_{jk}x_jx_k \\
+O\left(\sum_{1 \leq j \neq k \leq r} |x_jx_k| 
\left(\sum_{\substack{1 \leq i \leq r\\i \neq j,k}}|c_{ji}c_{ik}| 
+ \sum_{\substack{1 \leq g \neq h \leq r\\g,h \neq j,k}} |c_{jg}c_{gh}c_{hk}|
\right)+\frac{r}{\log^3 \frac{1}{\delta}} \sum_{1 \leq j \neq k \leq r} |x_jx_k|\right).
\end{gather*}
Using the AM-GM inequality, the error term here is
\begin{align} \label{eq:order}
\ll \left(\frac{r}{\log^2 \frac{1}{\delta}}+\frac{r^2 }{\log^3 \frac{1}{\delta}}\right)\|\boldsymbol{x}\|^2
\ll \frac{r}{\log^2 \frac{1}{\delta}} \cdot \|\boldsymbol{x}\|^2
\end{align}
and the main term here is
\begin{align*}
\ll \frac{\log 2r}{\log \frac{1}{\delta}}\|\boldsymbol{x}\|^2
\end{align*}
as in the proof of Lemma \ref{lem:cxx}. Combining with (\ref{eq:reciprocal}) and (\ref{eq:order}) followed by exponentiation, the proof is completed.
\end{proof}
\end{lemma}

We also need two large deviation lemmas.

\begin{lemma} \label{lem:decaymu}
Let $r \geq 1.$ Then for any $R > \sqrt{\delta \log \frac{1}{\delta}},$ we have
\begin{align*}
\mu_{\delta;\boldsymbol{t}}\left( \| \boldsymbol{x} \|_{\infty}>R \right) \leq 2r\exp \left( -\frac{R^2}{4\delta \log \frac{1}{\delta}} \right).
\end{align*}

\begin{proof}
The union bound gives
\begin{align*}
\mu_{\delta;\boldsymbol{t}}\left( \| \boldsymbol{x} \|_{\infty}>R \right)=\mathbb{P}( \left\| \boldsymbol{X}_{\delta,\boldsymbol{t}} \right\|_{\infty}>R ) \leq 
\sum_{j=1}^r \left(\mathbb{P}( X_{\delta,t_j} >R )+\mathbb{P}(  X_{\delta,t_j} <-R \right)).
\end{align*}
By symmetry, here we only bound the $\mathbb{P}\left( X_{\delta,t_j} >R \right)$. The moment generating function of $X_{\delta,t}$ is
\begin{align*}
\mathbb{E}( e^{sX_{\delta,t}} )
&=\prod_{\gamma>0} \mathbb{E}
\left(\exp \left( 2s\Re\left( w(\rho)U_{\gamma} \right) \right) \right)\\
&=\prod_{\gamma>0} I_0 \left( 2s\left| w(\rho) \right| \right),
\end{align*}
where 
\begin{align*}
I_0(x):=\frac{1}{2\pi}\int_{0}^{2\pi}e^{x\cos(\theta)}d\theta
\end{align*}
is the modified
Bessel function of order $0.$ Applying Chernoff bound with the inequality
$I_0(x) \leq \exp \left( x^2/4 \right)$
for $x \in \mathbb{R}$ (see \cite[Lemma 2.3]{MR2998146}), we have
\begin{align*}
\mathbb{P}\left(  X_{\delta,t} >R \right)
\leq e^{-sR} \, \mathbb{E}( e^{sX_{\delta,t}} )
\leq \exp \left( -sR+\frac{s^2}{2}V_j \right).
\end{align*}
Since $V_j = \delta\log\frac{1}{\delta}+O(\delta)$ by Proposition \ref{thm:cov}, the lemma follows from taking 
$s=R\left(\delta \log \frac{1}{\delta} \right)^{-1}.$ 
\end{proof}

\end{lemma}
\begin{lemma}\label{lem:decay2}
Let $1 \leq r \leq \frac{1}{2}\log \frac{1}{\delta}$ and $\boldsymbol{x} \in \mathbb{R}^r.$ Suppose  $R >  \sqrt{2r\log 2r}.$ Then
\begin{align*}
\frac{1}{(2\pi)^{r/2}(\det\mathcal{C})^{1/2}}
\int_{\|\boldsymbol{x}\|>R}\exp \left( -\frac{1}{2}\langle \mathcal{C}^{-1}\boldsymbol{x},\boldsymbol{x}\rangle \right) d\boldsymbol{x} \ll \exp \left( -\frac{R^2}{4} \right).
\end{align*}
\begin{proof}
Applying Lemma \ref{lem:finale}, the integral is
\begin{align*}
\ll \frac{1}{(2\pi)^{r/2}} \int_{\|\boldsymbol{x}\|>R}
\exp \left( -\frac{1}{2}\|\boldsymbol{x}\|^2 \left( 1+O\left(\frac{\log 2r}{\log \frac{1}{\delta}} \right) \right) \right) d\boldsymbol{x}.
\end{align*}
Then, as in the proof of Lemma \ref{lem:comp}, this is again
\begin{align*}
\ll \exp \left(-\frac{R^2}{4} \right),
\end{align*}
and hence the lemma follows.
\end{proof}
\end{lemma}

\begin{proof}[Proof of Theorem \ref{cor:compare}]
 Given $R>0,$ let
\begin{align*}
B_R:=
\left\{ \boldsymbol{x} \in B \,:\,
\|\boldsymbol{x}\|_{\infty} \leq R
\right\}
\end{align*}
and
\begin{align*}
\widetilde{B}_R
:=\left\{ \boldsymbol{x} \in \mathbb{R}^r \,:\, 
\left( \frac{x_1}{\sqrt{V_1}},\ldots,\frac{x_r}{\sqrt{V_r}}  \right) \in B_R
\right\}.
\end{align*}
Then, by Proposition \ref{thm:cov} and Lemma \ref{lem:decaymu}, we have
\begin{align} \label{eq:heybro2}
\mathbb{P}_x^{\log}( \widetilde{\boldsymbol{E}}(x;\delta,\boldsymbol{t}) \in B )
=\mu_{\delta,\boldsymbol{t}}( \widetilde{B}_R )
+O\left( r \exp \left( -\frac{R^2}{4} \right) \right).
\end{align}
It follows from Theorem \ref{thm:clt} that
\begin{gather} 
\mu_{\delta,\boldsymbol{t}}( \widetilde{B}_R )=\frac{1}{(2\pi)^{r/2}(\det\mathcal{C})^{1/2}}
\int_{B_R}
\exp \left( -\frac{1}{2}\langle \mathcal{C}^{-1}\boldsymbol{x},\boldsymbol{x}\rangle  \right)
d\boldsymbol{x} \nonumber\\
+O\left( r^4 T^3 \left( \frac{1}{\sqrt{2\pi}}
+O \left( \frac{\log 2r}{\log 1/\delta} \right) \right)^r \left( \frac{1}{\delta}\log \frac{1}{\delta} \right)^{-1}(2R)^r \right). \label{eq:afterthm12}
\end{gather}
Since $\|\boldsymbol{x}\|=\|\boldsymbol{x}\|_2 \geq \|\boldsymbol{x}\|_{\infty},$ applying Lemma \ref{lem:decay2} yields
\begin{gather} \label{eq:whatisgoingon2}
\frac{1}{(2\pi)^{r/2}(\det\mathcal{C})^{1/2}}
\int_{B \setminus B_R}
\exp \left( -\frac{1}{2}\langle \mathcal{C}^{-1}\boldsymbol{x},\boldsymbol{x}\rangle  \right)
d\boldsymbol{x} 
\ll  \exp\left( -\frac{R^2}{4} \right).
\end{gather}
Taking $R=\sqrt{5 \log \frac{1}{\delta}},$ then the theorem follows from combining (\ref{eq:heybro2}), (\ref{eq:afterthm12}) and (\ref{eq:whatisgoingon2}).
\end{proof}

\section{Proof of corollaries} \label{sec:pfofcor}

Applying Theorems \ref{thm:clt} and \ref{cor:compare}, we prove the remaining corollaries. 

\begin{proof}[Proof of  Corollary \ref{cor:negcorr}]
Throughout the proof, all implied constants depend on both $r$ and $T.$ Given $M>0,$ let
\begin{align*}
R_{\boldsymbol{\alpha},\boldsymbol{\beta};M}:=
\left\{ \boldsymbol{x} \in R_{\boldsymbol{\alpha},\boldsymbol{\beta}}\,:\,
\|\boldsymbol{x}\| \leq M
\right\}
\end{align*}
and
\begin{align*}
\widetilde{R}_{\boldsymbol{\alpha},\boldsymbol{\beta};M}
:=\left\{ \boldsymbol{x} \in \mathbb{R}^r \,:\, 
\left( \frac{x_1}{\sqrt{V_1}},\ldots,\frac{x_r}{\sqrt{V_r}}  \right) \in R_{\boldsymbol{\alpha},\boldsymbol{\beta};M}
\right\}.
\end{align*}
Since $\|\boldsymbol{x}\|=\|\boldsymbol{x}\|_2 \leq \sqrt{r} \|\boldsymbol{x}\|_{\infty},$  applying Lemma \ref{lem:decaymu} with  Proposition \ref{thm:cov} gives
\begin{align} \label{eq:heybro}
\mathbb{P}_x^{\log}( \widetilde{\boldsymbol{E}}(x;\delta,\boldsymbol{t}) \in R_{\boldsymbol{\alpha},\boldsymbol{\beta}} )
=\mu_{\delta,\boldsymbol{t}}( \widetilde{R}_{\boldsymbol{\alpha},\boldsymbol{\beta};M} )
+O\left( r \exp \left( -\frac{M^2}{4r} \right) \right).
\end{align}
Then, it follows from Theorem \ref{thm:clt} that
\begin{gather} 
\mu_{\delta,\boldsymbol{t}}( \widetilde{R}_{\boldsymbol{\alpha},\boldsymbol{\beta};M} )=\frac{1}{(2\pi)^{r/2}(\det\mathcal{C})^{1/2}}
\int_{R_{\boldsymbol{\alpha},\boldsymbol{\beta};M}}
\exp \left( -\frac{1}{2}\langle \mathcal{C}^{-1}\boldsymbol{x},\boldsymbol{x}\rangle  \right)
d\boldsymbol{x} \nonumber\\
+O\left( \left( \frac{1}{\delta}\log \frac{1}{\delta} \right)^{-1}M^r \right). \label{eq:afterthm1}
\end{gather}
Let $M=\left(\log \frac{1}{\delta}\right)^{1/10}.$ For $\|\boldsymbol{x}\| \leq M,$ 
applying Lemma \ref{lem:finale} followed by Taylor expansions
\begin{align*}
\exp\left( O\left( \frac{1}{\log^2 \frac{1}{\delta}}\|\boldsymbol{x}\|^2 \right) \right)=
1+O\left( \frac{1}{\log^2 \frac{1}{\delta}}\|\boldsymbol{x}\|^2 \right)
\end{align*}
and
\begin{align*}
\exp\left( \sum_{1\leq j<k\leq r}c_{jk}x_jx_k \right)
&=1+\sum_{1\leq j<k\leq r}c_{jk}x_jx_k 
+O\left(\left(\sum_{1\leq j<k\leq r}|c_{jk}||x_jx_k|\right)^2\right)\\
&=1+\sum_{1\leq j<k\leq r}c_{jk}x_jx_k
+O\left( \frac{1}{\log^2 \frac{1}{\delta}}\|\boldsymbol{x}\|^4 \right),
\end{align*}
we have
\begin{gather} 
\frac{1}{(2\pi)^{r/2}(\det\mathcal{C})^{1/2}}
\exp \left( -\frac{1}{2}\langle \mathcal{C}^{-1}\boldsymbol{x},\boldsymbol{x}\rangle  \right)=\frac{1}{(2\pi)^{r/2}}\exp\left( -\frac{1}{2}\|\boldsymbol{x}\|^2 \right) \nonumber\\
\cdot  
\left(1+\sum_{1\leq j<k \leq r}c_{jk}x_jx_k+O\left( \frac{1}{\log^2 \frac{1}{\delta}} 
+ \frac{1}{\log^2 \frac{1}{\delta}} \|\boldsymbol{x}\|^2
+ \frac{1}{\log^2 \frac{1}{\delta}}\|\boldsymbol{x}\|^4 \right)\right). \label{eq:idkidk}
\end{gather}
The contribution of the main term to (\ref{eq:afterthm1}) is
\begin{align} \label{eq:mtmtmt}
&\frac{1}{(2\pi)^{r/2}}
\int_{R_{\boldsymbol{\alpha},\boldsymbol{\beta};M}}\exp\left(-\frac{1}{2}\|\boldsymbol{x}\|^2\right)
\left( 1+\sum_{1\leq j<k \leq r}c_{jk}x_jx_k
 \right) d\boldsymbol{x}.
 \end{align}
On the other hand, we have
\begin{gather*}
\frac{1}{(2\pi)^{r/2}}
\int_{R_{\boldsymbol{\alpha},\boldsymbol{\beta}}\setminus R_{\boldsymbol{\alpha},\boldsymbol{\beta};M}}
\exp\left(-\frac{1}{2}\|\boldsymbol{x}\|^2\right)
\left( 1+\sum_{1\leq j<k \leq r}c_{jk}x_jx_k
 \right) d\boldsymbol{x}\\
\ll 
\int_{\|\boldsymbol{x}\|>M}
\|\boldsymbol{x}\|^2
 \exp \left( -\frac{1}{2}\|\boldsymbol{x}
\|^2\right)dx
\ll  \exp \left( -\frac{M^2}{4} \right).
\end{gather*}
Therefore, the expression (\ref{eq:mtmtmt}) is 
 \begin{gather} 
\Phi(R_{\boldsymbol{\alpha},\boldsymbol{\beta}})+\frac{1}{2\pi}\sum_{1 \leq j<k \leq r} 
c_{jk}
\left( e^{ -\frac{1}{2}\alpha_j^2}-e^{ -\frac{1}{2}\beta_j^2}\right)\left( e^{-\frac{1}{2}\alpha_k^2}-e^{-\frac{1}{2}\beta_k^2}\right)
\Phi\left(\prod_{\substack{i=1\\i \neq j,k}}^r  (\alpha_i,\beta_i]\right) \nonumber\\
+O\left( \exp\left(-\frac{M^2}{4} \right) \right). \label{eq:mainalmost}
\end{gather}
Meanwhile, the contribution of the error term of (\ref{eq:idkidk}) to (\ref{eq:afterthm1}) is
\begin{align} \label{eq:etalmostfinale}
\ll \int_{R_{\boldsymbol{\alpha},\boldsymbol{\beta}}}\exp\left( -\frac{1}{2}\|\boldsymbol{x}\|^2 \right)
\left( \frac{1}{\log^2 \frac{1}{\delta}} 
+ \frac{1}{\log^2 \frac{1}{\delta}} \|\boldsymbol{x}\|^2
+\frac{1}{\log^2 \frac{1}{\delta}}\|\boldsymbol{x}\|^4 \right) d\boldsymbol{x}
\ll \frac{1}{\log^2 \frac{1}{\delta}}.
\end{align}
Finally, by Proposition \ref{thm:cov} with Lemma \ref{lem:eq:coulomb}, we have
\begin{align*}
c_{jk}=\frac{\Cov_{jk}}{\sqrt{V_jV_k}}
=-\frac{\Delta(|t_j-t_k|)}{\log \frac{1}{\delta}}
+O\left( \frac{1}{|t_j-t_k|} \cdot \frac{1}{\log^2 \frac{1}{\delta}} \right).
\end{align*}
Substituting this into (\ref{eq:mainalmost}) gives
\begin{align*}
\Phi(R_{\boldsymbol{\alpha},\boldsymbol{\beta}})
 -\frac{1}{2\pi\log \frac{1}{\delta}}\sum_{1 \leq j<k \leq r} 
 \Delta(|t_j-t_k|)
( e^{ -\frac{1}{2}\alpha_j^2}-e^{ -\frac{1}{2}\beta_j^2})( e^{-\frac{1}{2}\alpha_k^2}-e^{-\frac{1}{2}\beta_k^2})
\Phi\left(\prod_{\substack{i=1\\i \neq j,k}}^r  (\alpha_i,\beta_i]\right)
\end{align*}
with an error term
\begin{align} \label{eq:etfinale}
\leq \frac{1}{\log^2 \frac{1}{\delta}}
\sum_{1\leq j<k \leq r}\frac{1}{|t_j-t_k|}
&= \frac{1}{\log^2 \frac{1}{\delta}}\sum_{1 \leq j \leq r}
\sum_{\substack{1 \leq k \leq r\\k \neq j}}\frac{1}{|t_j-t_k|} \nonumber\\
&\ll \frac{1}{\log^2 \frac{1}{\delta}}
\end{align}
as $|t_j-t_k|\geq 1$ for $1 \leq j \neq k \leq r.$ Collecting the error terms in (\ref{eq:heybro}), (\ref{eq:afterthm1}), (\ref{eq:mainalmost}), (\ref{eq:etalmostfinale}) and (\ref{eq:etfinale}) with $M=\left( \log \frac{1}{\delta} \right)^{1/10}$, the corollary follows.
\end{proof}

\begin{proof}[Proof of Corollary 
\ref{cor:largedev}]
Arguing analogously, one can show that
\begin{gather*}
\mathbb{P}_x^{\log}
\left( \|\widetilde{\boldsymbol{E}}(x;\delta,\boldsymbol{t})\| >V \right)=
\frac{1}{(2\pi)^{r/2}} \int_{\|\boldsymbol{x}\|>V}
\exp \left( -\frac{1}{2}\|\boldsymbol{x}\|^2 \right) d\boldsymbol{x} \\
-\frac{1}{\log \frac{1}{\delta}} \sum_{1 \leq j<k \leq r} 
\frac{\Delta (|t_j-t_k|)}{(2\pi)^{r/2}}
\int_{\|\boldsymbol{x}\|>V} x_jx_k e^{-\frac{1}{2}\|\boldsymbol{x}\|^2} d\boldsymbol{x}
+O_{r,T}\left( \frac{1}{\log^2 \frac{1}{\delta}} \right).
\end{gather*}
Since the above integrals 
\begin{align*}
\int_{\|\boldsymbol{x}\|>V} x_jx_k e^{-\frac{1}{2}\|\boldsymbol{x}\|^2} d\boldsymbol{x}
\end{align*}
vanish for $1 \leq j<k \leq r$ by symmetry, the corollary follows. 
\end{proof}

\begin{proof}[Proof of Corollary \ref{thm:extremebias}]
An analogous argument to that of Corollary \ref{cor:negcorr} applies. To avoid repetition, the proof is omitted.
\end{proof}

\begin{proof}[Proof of  Corollary \ref{cor:bias}]

Appealing to Theorem \ref{cor:compare}, we have
\begin{gather*} 
\rho(\delta;\boldsymbol{t}) 
=
\frac{1}{(2\pi)^{r/2}(\det\mathcal{C})^{1/2}}\int_{x_1>\cdots>x_r}
\exp \left( -\frac{1}{2}\langle \mathcal{C}^{-1}\boldsymbol{x},\boldsymbol{x}\rangle\right) d\boldsymbol{x} 
 \\
+O \left( r^4 T^3 \delta \left( 4\log \frac{1}{\delta} \right)^{r/2-1} \right).
\end{gather*}
Using Stirling's formula, one can show that the error term here is negligible.
Applying Lemma \ref{lem:finale}, the integral here is
\begin{gather}
\left( 1+O\left( \frac{r}{\log^2 \frac{1}{\delta}} \right) \right)
\frac{1}{(2\pi)^{r/2}}\int_{x_1>\cdots>x_r}
\exp \left( -\frac{1}{2}\|\boldsymbol{x}\|^2
\left( 1+O\left( \frac{\log 2r}{\log \frac{1}{\delta}} \right) \right)\right) d\boldsymbol{x} \nonumber\\
=\left( 1+O\left( \frac{r}{\log^2 \frac{1}{\delta}} \right) \right)
\frac{1}{(2\pi)^{r/2}}\int_{x_1>\cdots>x_r}
\exp \left( -\frac{1}{2}\|\boldsymbol{x}\|^2 \right) d\boldsymbol{x} \nonumber \\
+ O \left( 
\frac{1}{(2\pi)^{r/2}}\int_{x_1>\cdots>x_r}
\exp \left( -\frac{1}{2}\|\boldsymbol{x}\|^2 \right) 
\left| 
\exp \left( C \cdot \frac{\log 2r}{\log \frac{1}{\delta}}\|\boldsymbol{x}\|^2 \right) -1
\right| d\boldsymbol{x}
\right) \nonumber \\
+ O \left( 
\frac{1}{(2\pi)^{r/2}}\int_{x_1>\cdots>x_r}
\exp \left( -\frac{1}{2}\|\boldsymbol{x}\|^2 \right) 
\left| 
\exp \left( -C \cdot \frac{\log 2r}{\log \frac{1}{\delta}}\|\boldsymbol{x}\|^2 \right) -1
\right| d\boldsymbol{x}
\right) \label{eq:expc}
\end{gather}
for some absolute constant $C>0.$ By symmetry, the first term here is
\begin{align} \label{eq:fhlmt}
\left( 1+O\left( \frac{r}{\log^2 \frac{1}{\delta}} \right) \right) \frac{1}{r!}.
\end{align}
By Taylor expanding the exponential function, the second and third terms of (\ref{eq:expc}) are
\begin{align*}
\ll \sum_{k=1}^{\infty} \frac{1}{k!} \left( C \cdot \frac{\log 2r}{\log \frac{1}{\delta}} \right)^k \frac{1}{(2 \pi)^{r/2}} \int_{x_1>\cdots>x_r} 
\exp \left( -\frac{1}{2}\|\boldsymbol{x}\|^2 \right) \|\boldsymbol{x}\|^{2k} d \boldsymbol{x},
\end{align*}
which is again by symmetry
\begin{align} \label{eq:aftertaylorsymm}
=\frac{1}{r!}\sum_{k=1}^{\infty} \frac{1}{k!} \left( C \cdot \frac{\log 2r}{\log \frac{1}{\delta}} \right)^k \frac{1}{(2 \pi)^{r/2}} \int_{\mathbb{R}^r} 
\exp \left( -\frac{1}{2}\|\boldsymbol{x}\|^2 \right) \|\boldsymbol{x}\|^{2k} d \boldsymbol{x}.
\end{align}
Using spherical coordinates, one can show that 
\begin{align*}
\frac{1}{(2 \pi)^{r/2}} \int_{\mathbb{R}^r} 
\exp \left( -\frac{1}{2}\|\boldsymbol{x}\|^2 \right) \|\boldsymbol{x}\|^{2k} d \boldsymbol{x}
=\frac{2\pi^{r/2}}{\Gamma(r/2)} \cdot \frac{1}{(2\pi)^{r/2}} 
\int_0^{\infty} e^{-\frac{1}{2}x^2} x^{r+2k-1} dx.
\end{align*}
Making the change of variables $y=\frac{1}{2}x^2,$ this becomes
\begin{align*}
 \frac{2\pi^{r/2}}{\Gamma(r/2)} \cdot \frac{1}{(2\pi)^{r/2}}   \int_{0}^{\infty}
 e^{-y} (2y)^{(r/2)+k-1} dy
=
2^k \cdot  \frac{\Gamma((r/2)+k-1)}{\Gamma(r/2)}.
\end{align*}
Plugging this into (\ref{eq:aftertaylorsymm}), the second and third terms of (\ref{eq:expc}) are
\begin{align}
\ll \frac{1}{r!} \sum_{k=1}^{\infty}   \left( 2C \cdot \frac{\log 2r}{\log \frac{1}{\delta}} \right)^k \frac{\Gamma((r/2)+k-1)}{\Gamma(r/2)\Gamma(k+1)} 
\ll& \frac{1}{r!} \sum_{k=1}^{\infty}   \left( 2C \cdot \frac{r\log 2r}{\log \frac{1}{\delta}} \right)^k \nonumber\\
\ll& \frac{1}{r!} \cdot \frac{r\log 2r}{\log \frac{1}{\delta}}. \label{eq:fhlet}
\end{align}
Combining (\ref{eq:fhlmt}) and (\ref{eq:fhlet}), the corollary follows.
\end{proof}

\begin{proof}[Proof of Corollary \ref{cor:extremebias}]


Let $R=10\sqrt{2r\log 2r}.$ We define
\begin{align*}
H_{r,s}:=\{  x_1>\cdots>x_s>\max_{s<j \leq r}x_j \},
\end{align*}
\begin{align*}
H_{r,s;R}:=\{ x_1>\cdots>x_s>\max_{s<j \leq r}x_j, \| \boldsymbol{x} \|_{\infty} \leq R\}
\end{align*}
and
\begin{align*}
\widetilde{H}_{r,s;R}:=\left\{ \boldsymbol{x} \in \mathbb{R}^r \,:\, 
\left( \frac{x_1}{\sqrt{V_1}},\ldots,\frac{x_r}{\sqrt{V_r}}  \right) \in H_{r,s;R}
\right\}.
\end{align*}
Applying
Lemma \ref{lem:decaymu} with Proposition \ref{thm:cov} gives
\begin{align} \label{eq:2024et1}
\rho_s(\delta;\boldsymbol{t})=
\mu_{\delta;\boldsymbol{t}}(\widetilde{H}_{r,s;R})
+O \left(r\exp \left( -\frac{R^2}{4} \right) \right).
\end{align}
Then, it follows from Theorem \ref{thm:clt} that
\begin{gather} 
\mu_{\delta;\boldsymbol{t}}(\widetilde{H}_{r,s;R})=\frac{1}{(2\pi)^{r/2}(\det\mathcal{C})^{1/2}}
\int_{H_{r,s;R}}
\exp \left( -\frac{1}{2}\langle \mathcal{C}^{-1}\boldsymbol{x},\boldsymbol{x}\rangle  \right)
d\boldsymbol{x} \nonumber\\
+O\left( \frac{r^4T^3}{(2\pi)^{r/2}}   \left(\frac{1}{\delta}
\log \frac{1}{\delta} \right)^{-1} \meas(H_{r,s;R}) \right). \label{eq:2024mt1}
\end{gather}
Note that by symmetry, we have
\begin{align*}
\meas(H_{r,s;R})=\frac{(r-s)!}{r!} \cdot (2R)^r,
\end{align*}
so that the error term here is
\begin{align} \label{eq:2024et2}
\ll \frac{(r-s)!}{r!} \cdot \frac{r^4T^3}{(2\pi)^{r/2}}   \left(\frac{1}{\delta}
\log \frac{1}{\delta} \right)^{-1} (2R)^r.
\end{align}
On the other hand, since $\|\boldsymbol{x}\|_{\infty} \leq \|\boldsymbol{x}\|=\|\boldsymbol{x}\|_2,$ we have
\begin{align*}
H_{r,s} \cap \{ \|\boldsymbol{x}\| \leq R \} = H_{r,s;R} \cap \{ \|\boldsymbol{x}\| \leq R \}.
\end{align*}
Therefore, by Lemma \ref{lem:decay2}, the main term of (\ref{eq:2024mt1}) is
\begin{align} \label{eq:2024et3}
\frac{1}{(2\pi)^{r/2}(\det\mathcal{C})^{1/2}}
\int_{H_{r,s} \cap \{\|\boldsymbol{x}\|\leq R\}}
\exp \left( -\frac{1}{2}\langle \mathcal{C}^{-1}\boldsymbol{x},\boldsymbol{x}\rangle  \right)
d\boldsymbol{x}
+O\left( \exp \left( -\frac{R^2}{4} \right) \right).
\end{align}
Appealing to Lemma \ref{lem:finale}, the main term here is
\begin{align} \label{eq:2024mt2}
\frac{1+o(1)}{(2\pi)^{r/2}}
\int_{H_{r,s} \cap \{\|\boldsymbol{x}\|\leq R\}}
\exp\left(-\frac{1}{2}\left( 1+O\left( 
\frac{r}{\log^2 \frac{1}{\delta}} \right) \right)\|\boldsymbol{x}\|^2+\sum_{1 \leq j<k\leq r}c_{jk}x_jx_k \right)
d\boldsymbol{x} 
\end{align}
Using the assumption that $|t_j-t_k| \geq \log \frac{1}{\delta}$ whenever $\max\{j,k \}>s$ for $1 \leq j \neq k \leq r,$ it follows from Proposition \ref{thm:cov} that
\begin{align*}
\sum_{1 \leq j<k\leq r} c_{jk}x_jx_k =
-\frac{1}{\log \frac{1}{\delta}}\sum_{1 \leq j<k\leq s} \Delta(|t_j-t_k|) x_jx_k
+O\left( \frac{r }{\log^2 \frac{1}{\delta}} \|\boldsymbol{x}\|^2\right).
\end{align*}
Since $\|\boldsymbol{x}\| \leq R=10\sqrt{2r\log 2r},$ we conclude that (\ref{eq:2024mt2}) is
\begin{align} \label{eq:2024mt3}
\frac{1+o(1)}{(2\pi)^{r/2}}
\int_{H_{r,s} \cap \{\|\boldsymbol{x}\|\leq R\}}
\exp\left(-\frac{1}{2}\|\boldsymbol{x}\|^2-\frac{1}{\log \frac{1}{\delta}}\sum_{1 \leq j<k\leq s} \Delta(|t_j-t_k|) x_jx_k \right)
d\boldsymbol{x}.
\end{align}

To detect a strong bias, let $M=\sqrt{(2-\varepsilon)\log (r/s)}.$ Then, again by Proposition \ref{thm:cov}, the contribution from those $\boldsymbol{x} \in \mathbb{R}^r$ for which $ \max_{s<j \leq r} x_j \leq M$ to (\ref{eq:2024mt3}) is
\begin{gather}  
\ll \frac{1}{(2\pi)^{r/2}}
\int_{H_{r,s} \cap \{\|\boldsymbol{x}\|\leq R\} \cap \{ \max_{s<j \leq r} x_j \leq M \}}
\exp \left( 
-\frac{1}{2}  
\left( 1+O \left( \frac{\log 2r}{\log \frac{1}{\delta}} \right) \right)
\|\boldsymbol{x}\|^2 
\right) d\boldsymbol{x} \nonumber\\
\ll \frac{1}{(2\pi)^{r/2}} 
\int_{H_{r,s}\cap \{ \max_{s<j \leq r} x_j \leq M \}}
\exp \left( -\frac{1}{2}\|\boldsymbol{x}\|^2 \right) d\boldsymbol{x}. \label{eq:2024et4}
\end{gather}
We further decompose the domain of integration as
\begin{align*}
H_{r,s}\cap \{ \max_{s<j \leq r} x_j \leq M \}
=H_0 \sqcup H_1 \cdots \sqcup H_s,
\end{align*}
where
\begin{align*}
H_i:= \{ x_1>\cdots>x_i>M \geq x_{i+1} >\cdots > x_s>\max_{s<j \leq r} x_j \}.
\end{align*}
Then, the integral over each $H_i$ is computable and is
\begin{align*}
\frac{1}{(2\pi)^{r/2}}
\int_{H_i} \exp \left( -\frac{1}{2}\|\boldsymbol{x}\|^2 \right) d\boldsymbol{x}
&=\frac{1}{i!}(1-\Phi(M))^i \cdot \frac{(r-s)!}{(r-i)!}\Phi(M)^{r-i} \\
&= \frac{(r-s)!}{r!} \cdot  {r \choose i} (1-\Phi(M))^i
\Phi(M)^{r-i},
\end{align*}
where 
\begin{align*}
\Phi(M):=\frac{1}{\sqrt{2\pi}}\int_{-\infty}^M e^{-\frac{1}{2}x^2} dx.
\end{align*}
Therefore, the expression (\ref{eq:2024et4}) is
\begin{align} \label{eq:2024et5}
\ll \frac{(r-s)!}{r!} \sum_{0 \leq i \leq s} 
{r \choose i} (1-\Phi(M))^i \Phi(M)^{r-i}.        
\end{align}

To bound the sum, let $X \sim b(n,p)$ denote a binomial random variable with parameters $n$ and $p.$ If $0 \leq q<p,$ then the Chernoff bound gives
\begin{align*}
\mathbb{P}(X \leq nq) =\sum_{0 \leq k \leq nq} {n \choose k}p^{k}(1-p)^{n-k} \leq e^{-nD(q \| p)},
\end{align*}
where 
\begin{align*}
D(q \| p):=q\log \frac{q}{p}+(1-q)\log \frac{1-q}{1-p}
\end{align*}
is the Kullback--Leibler divergence (see \cite[Corollary 4.1]{MR3793013} for instance).
Therefore, we have
\begin{gather}
\sum_{0 \leq i \leq s} {r \choose i}(1-\Phi(M))^i \Phi(M)^{r-i} \nonumber \\
\leq \exp \left( -s\log \frac{s/r}{1-\Phi(M)}-(r-s)\log \frac{1-(s/r)}{\Phi(M)}  \right). \label{eq:afterkl}
\end{gather}
Recall that $M=\sqrt{(2-\varepsilon)\log (r/s)}.$ Let $\eta=\eta(\varepsilon)>0$ be sufficiently small so that
\begin{align*}
1-\Phi(M) =&
\frac{1}{\sqrt{2\pi}}\int_{M}^{\infty} e^{-\frac{1}{2}x^2} dx \\
\geq&  \frac{1}{10M} \exp\left( -\frac{1}{2}M^2 \right) \\
>& 100 \eta.
\end{align*}
Then, one can show that \eqref{eq:afterkl} is bounded by 
$\exp(-3s),$ so that (\ref{eq:2024et5}) is
\begin{align}
\leq \frac{(r-s)!}{r!} \cdot e^{-3s}. \label{eq:exp-3s}
\end{align}
Finally, we are left with the integral
\begin{align*}
\frac{1+o(1)}{(2\pi)^{r/2}}
\int_{H_{r,s} \cap \{\|\boldsymbol{x}\|\leq R\} \cap 
\{ \max_{s<j \leq r} x_j > M \}}
\exp\left(-\frac{1}{2}\|\boldsymbol{x}\|^2-\frac{1}{\log \frac{1}{\delta}}\sum_{1 \leq j<k\leq s} \Delta(|t_j-t_k|) x_jx_k \right)
d\boldsymbol{x},
\end{align*}
which is 
\begin{align*}
\leq 
\exp \left( o(1)
-\frac{M^2}{\log \frac{1}{\delta}} 
\sum_{1 \leq j<k\leq s} \Delta(|t_j-t_k|) \right)
\frac{1}{(2\pi)^{r/2}}
\int_{H_{r,s} }
e^{-\frac{1}{2}\|\boldsymbol{x}\|^2} d\boldsymbol{x}.
\end{align*}
Substituting $M=\sqrt{(2-\varepsilon)\log (r/s)},$ this is
\begin{align*}
\exp \left( o(1)
-(2-\varepsilon) \cdot \frac{\log(r/s)}{\log \frac{1}{\delta}} 
\sum_{1 \leq j<k\leq s} \Delta(|t_j-t_k|) \right)
 \frac{(r-s)!}{r!}.
\end{align*}
Using Stirling's formula, one can show that the error terms in (\ref{eq:2024et1}), (\ref{eq:2024et2}),  (\ref{eq:2024et3}) and \eqref{eq:exp-3s} are negligible, and hence the corollary follows.
\end{proof}

\begin{proof}[Proof of Corollary \ref{cor:extremebiasii}]

Consider the configuration
\begin{align*}
\boldsymbol{u}=\left(1,2,\ldots,s,-\log \frac{1}{\delta}, -2\log \frac{1}{\delta}, \ldots, -(r-s)\log \frac{1}{\delta} \right)
\end{align*}
for some integer $1 \leq s \leq r.$ Note that $T_{[r]} \leq \delta^{-\frac{1}{10}}$ for sufficiently small $\delta>0.$ Let $\varepsilon=1$ and $s=\lfloor \eta r \rfloor,$ where $\eta>0$ is the absolute constant from Corollary \ref{cor:extremebias}. Then, it follows from Lemma \ref{lem:eq:coulomb} that
\begin{align*}
\sum_{1 \leq j < k \leq s}\Delta(|u_j-u_k|)
=\left( \frac{1}{2}+o(1) \right) s\log s.
\end{align*}
Therefore, applying Corollary \ref{cor:extremebias} gives
\begin{align*}
\rho_s(\delta;\boldsymbol{u})
\leq \exp \left( o(1)-\frac{ \log (1/\eta)}{\log \frac{1}{\delta}} \cdot \frac{\eta r}{4} \log \frac{\eta r}{2}\right) \frac{(r-s)!}{r!},
\end{align*}
so that there exists an absolute constant $\eta_0>0$ for which
\begin{align} \label{eq:2024u1}
\rho_s(\delta;\boldsymbol{u}) \leq \exp \left( -\eta_0 \cdot  \frac{r\log \log \frac{1}{\delta}}{\log \frac{1}{\delta}} \right)\frac{(r-s)!}{r!},
\end{align}
provided that $\delta>0$ is sufficiently small. By symmetry, we have
\begin{align*}
\rho_s(\delta;\boldsymbol{u}) 
=\sum_{\sigma \in S_{r-s}} 
\rho(\delta; u_1, u_2, \ldots, u_s, u_{\sigma(s+1)}, u_{\sigma(s+2)}, \ldots, u_{\sigma(r)}),
\end{align*}
where $S_{r-s}$ is regarded as the symmetric group on the set $\{ s+1, s+2, \ldots, r \}.$  Therefore, it follows from (\ref{eq:2024u1}) that there exists $\sigma \in S_{r-s}$ for which
\begin{align*}
\rho_s(\delta;u_1, u_2, \ldots, u_s, u_{\sigma(s+1)}, u_{\sigma(s+2)}, \ldots, u_{\sigma(r)}) \leq \exp \left( -\eta_0 \cdot  \frac{r\log \log \frac{1}{\delta}}{\log \frac{1}{\delta}} \right) \frac{1}{r!},
\end{align*}
which yields the corollary on taking $\boldsymbol{t}=(u_1, u_2, \ldots, u_s, u_{\sigma(s+1)}, u_{\sigma(s+2)}, \ldots, u_{\sigma(r)}).$
\end{proof}


\section{Primes in shorter intervals} \label{newsection}


One unfortunate drawback of the Rubinstein--Sarnak approach is the qualitative nature of the LI hypothesis, which prevents achieving uniformity of $\delta$ in $X$ for $x \in [X, 2X].$ Therefore, to address shorter intervals, i.e., $\delta=\delta(X) \to 0$ as $X \to \infty$, we shall propose the following conjecture.

\begin{conjecture}[Quantitative linear independence conjecture (QLI)] \label{conj: qli}
Let $k \geq 2.$ Then there exists a constant $c_k > k$ such that 
for any $\boldsymbol{\varepsilon} \in \{\pm 1\}^k,$ we have
\begin{align} \label{eq:offoffdiag}
\#\{ \boldsymbol{\gamma} \in [0,T]^k \, : \, 0<|\langle \boldsymbol{\varepsilon}, \boldsymbol{\gamma} \rangle |
\leq T^{-c_k}\}=o_k ( N(T)^{k/2} ) 
\end{align}
as $T \to \infty,$ where 
\begin{align*}
\langle \boldsymbol{\varepsilon}, \boldsymbol{\gamma} \rangle := \sum_{j=1}^k \varepsilon_j \gamma_j.
\end{align*}
\end{conjecture}
The heuristic is as follows. To justify our choice of $T^{-c_k}$ with $c_k>k,$ note that the average gap of 
 $\langle \boldsymbol{\varepsilon}, \boldsymbol{\gamma} \rangle \Mod{1}$ is $(2N(T))^{-k} \asymp_k (T\log T)^{-k}.$ On the other hand,
 assuming LI, the alternating sum  $\langle \boldsymbol{\varepsilon}, \boldsymbol{\gamma} \rangle$
vanishes if and only if $k$ is even and $[k]$ can be partitioned into $k/2$ pairs $\{j,j'\} $
of which $\varepsilon_{j'}=-\varepsilon_{j}$ and $\gamma_{j'}=\gamma_j.$ We call these $\boldsymbol{\gamma}$ diagonal. As there are $\asymp_k N(T)^{k/2}$ such vectors if $k$ is even and none if $k$ is odd, the conjecture can be reinterpreted as saying that $100 \%$ of the $\boldsymbol{\gamma}$ satisfying 
$
|\langle \boldsymbol{\varepsilon}, \boldsymbol{\gamma} \rangle |
\leq T^{-c_k}$
are diagonal. 

For future reference, let $g(T) \leq \log T$ be a positive increasing function satisfying $g(T) \to \infty$ as $T \to \infty$ for which the left-hand side of (\ref{eq:offoffdiag}) is 
\begin{align} \label{eq:g(T)}
\ll_k N(T)^{k/2}/g(T).
\end{align}

\begin{remark}
Interestingly, Lamzouri \cite{lamzouri2023effective} very recently formulated an \textit{effective linear independence conjecture} (ELI), which is stronger than our QLI to obtain omega results for the error term in the prime number theorem. They are conjectured to be best possible by Montgomery \cite[Lecture 3]{MR634679}. For a weaker formulation of LI, see \cite{MR4082248}.
\end{remark}

With QLI, we can now show that the weighted count of primes in a short moving interval of length \( h = h(x) = \delta x \) is asymptotically normal, as \( x \in [1, X] \) varies (in logarithmic scale), provided that \( \delta = \delta(X) > (\log X)^{-\varepsilon} \) for any \( \varepsilon > 0 \).

\begin{theorem} \label{thm:short} Assume RH, LI and QLI. Given real numbers $U, \delta>0$ for which  $\delta=o(1)$ but $\log (1/\delta)=o(\log U)$ as $U \to \infty.$ 
If $u \in [U, 2U]$ is chosen uniformly at random, then
as  $U \to \infty,$ we have convergence in distribution to a standard Gaussian
\begin{align*}
\widetilde{E}\left( e^u;\delta,0 \right)
 \xrightarrow[]{d} \mathcal{N}(0,1),
\end{align*}
where we recall that 
\begin{align*}
\widetilde{E}\left( x;\delta,0 \right)=\frac{1}{\sqrt{V(\delta,0)\cdot x}} 
\left(  \psi\left( x+\frac{1}{2}\delta x \right) - \psi\left( x-\frac{1}{2}\delta x \right) -\delta x  \right)
\end{align*}
with $V(\delta,0)=\sum_{\gamma} |w(\rho)|^2,$
i.e., for any fixed real numbers $\alpha<\beta,$ we have
\begin{align*}
\lim_{U \to \infty}\frac{1}{U}\mathop{\mathrm{meas}}\left\{ x \in [U,2U] \,: \, \widetilde{E}\left( e^u;\delta,0 \right) \in (\alpha,\beta] \right\}
=\frac{1}{\sqrt{2\pi}}\int_{\alpha}^{\beta}e^{-\frac{t^2}{2}}dt.
\end{align*}
\end{theorem}

\begin{remark}
By Proposition \ref{thm:cov}, in the above theorem $\widetilde{E}\left( x;\delta,0 \right)$
can be replaced by 
\begin{align*}
\frac{1}{\sqrt{\left( \delta \log \frac{1}{\delta}+(1-\gamma-2\pi)\delta \right)x}}
\left( \psi\left( x+\frac{1}{2}\delta x \right) - \psi\left( x-\frac{1}{2}\delta x \right) -\delta x \right).
\end{align*}
\end{remark}

Prior to the proof of the theorem, we compute moments of the finite approximation
\begin{align*}
\widetilde{E}^{(T)}\left( e^u;\delta,0 \right):=\frac{E^{(T)}\left( e^u;\delta,0 \right)}{\sqrt{V^{(T)}(\delta,0)}},
\end{align*}
where we recall that
\begin{align*}
E^{(T)}\left(e^u;\delta,0\right)=&-\sum_{|\gamma|\leq T}w(\rho)e^{i\gamma u}
\end{align*}
and
\begin{align*}
V^{(T)}(\delta,0):=&\sum_{|\gamma| \leq T}\left|w(\rho) \right|^2
\end{align*}
for some $T \geq 2$ such that $\delta T \to \infty$ as $U \to \infty.$

\begin{lemma} \label{lem:VT}
Let $\delta>0, T \geq 2.$ Then 
\begin{align*}
V^{(T)}(\delta,0)=V(\delta,0)+O\left( \frac{\log T}{T} \right).
\end{align*}

\begin{proof}
Since $w(\rho) \ll |\gamma|^{-1},$ we have
\begin{align*}
\sum_{|\gamma|>T} |w(\rho)|^2 &\ll 
\sum_{|\gamma|>T} \frac{1}{\gamma^2}\\
&\ll \int_{T}^{\infty} \frac{\log t}{t^2} dt \\
&\ll \frac{\log T}{T},
\end{align*}
and hence the lemma follows.
\end{proof}

\end{lemma}

Instead of the Lebesgue measure, the moments are computed with respect to a nonnegative Schwartz function $W$ supported on $\left(1/2, 5/2\right)$ of unit mass (see \cite{MR2039790} for instance). In particular, both $W$ and its Fourier transform $\widehat{W}$ are rapidly decaying in the sense that for any $A>0,$ we have
\begin{align} \label{eq:growthcond}
W(u), \, \widehat{W}(u) \ll_A (1+|u|)^{-A},
\end{align}
where the Fourier transform $\widehat{W}$ is defined as
\begin{align*}
\widehat{W}(v):=\int_{-\infty}^{\infty}W(u)e^{-iuv}du.
\end{align*}
For convenience, we write
\begin{align*}
\mathbb{E}_{u \sim U}^W\left( f(u) \right):= \frac{1}{U}\int_{-\infty}^{\infty} f(u)W\left( \frac{u}{U} \right)
du.
\end{align*}

\begin{proposition} \label{prop:moments}
Let $k \geq 1$ be an integer and $\delta \to 0^{+}$ as $U \to \infty.$ If $T \geq 2$ satisfies $\delta T \to \infty$ but $\delta T \leq g(T)^{1/k}$ as $U \to \infty,$ where $g(T)$ is defined in (\ref{eq:g(T)}), and $T \leq U^{(1-\varepsilon)/c_k}$ for some $\varepsilon>0,$ then
\begin{align*}
\mathbb{E}_{u \sim U}^W ( {\widetilde{E}^{(T)}\left( e^u;\delta,0 \right)}^{k} )
=\mu_k+o_k(1),
\end{align*}
where
\begin{align*}
\mu_{k}:=
\begin{cases}
\frac{k!}{2^{k/2}(k/2)!} & \mbox{{\normalfont if $k$ is even,} }\\
\hfil 0 & \mbox{{\normalfont otherwise. }} 
\end{cases}
\end{align*}
\end{proposition}

\begin{remark}
Without assuming LI but only RH, La Bretèche and Fiorilli \cite{MR4322621} were able to establish lower bounds for all even moments of suitably weighted (decaying exponentially with nonnegative Fourier transform) prime count in a short moving interval, provided that $\delta \in (0,1/2).$
\end{remark}

\begin{proof}
Fix $\varepsilon>0.$ Let $T \geq 2$ satisfying $\delta T \to \infty$ and $T \leq U^{(1-\varepsilon)/c_k}.$
By definition, we have
\begin{align*}
\mathbb{E}_{u \sim U}^W( {{E}^{(T)}\left( e^u;\delta,0 \right)}^{k} ) 
=&\frac{1}{U} \int_{-\infty}^{\infty} 
\left( -\sum_{0 < \gamma \leq T} ( w(\rho)e^{i\gamma u}+\overline{w(\rho)}e^{-i\gamma u} ) \right)^{k}W\left( \frac{u}{U} \right) du\\
=&(-1)^k\sum_{\boldsymbol{\varepsilon} \in \{ \pm 1\}^k }\sum_{\boldsymbol{\gamma}\in [0,T]^k} w^{\boldsymbol{\varepsilon}}(\boldsymbol{\rho}) \cdot
\widehat{W} \left( -U \langle \boldsymbol{\varepsilon}, \boldsymbol{\gamma} \rangle \right),
\end{align*}
where 
\begin{align*}
w^{\boldsymbol{\varepsilon}}(\boldsymbol{\rho}):=\prod_{\varepsilon_i=1}w(\rho_i)\prod_{\varepsilon_j=-1}\overline{w(\rho_j)}.
\end{align*}
We split the sum into $\Sigma_1, \Sigma_{2}$ and $\Sigma_{3}$ consisting of tuples $(\boldsymbol{\varepsilon}, \boldsymbol{\gamma})$ of which 
$\langle \boldsymbol{\varepsilon}, \boldsymbol{\gamma} \rangle=0,  
0<\left|\langle \boldsymbol{\varepsilon}, \boldsymbol{\gamma} \rangle \right| \leq T^{-c_k}$
and $\left|\langle \boldsymbol{\varepsilon}, \boldsymbol{\gamma} \rangle \right|>T^{-c_k}$
respectively, where $c_k$ is defined in Conjecture \ref{conj: qli}.

Let us deal with the sum $\Sigma_3$ first. Using (\ref{eq:growthcond}), we have
\begin{align*}
\widehat{W} \left( -U \langle \boldsymbol{\varepsilon}, \boldsymbol{\gamma} \rangle \right) 
&\ll_A \left( 1+U\left| \langle \boldsymbol{\varepsilon}, \boldsymbol{\gamma} \rangle \right| \right)^{-A} \\
&\ll_A  \left( T^{-c_k}U \right)^{-A},
\end{align*}
so that the sum $\Sigma_3$ is
\begin{align*}
\ll_A \left( T^{-c_k}U \right)^{-A}\underset{\left| \langle \boldsymbol{\varepsilon}, \boldsymbol{\gamma} \rangle \right|>T^{-c_k}}{\sum\sum} \left| w^{\boldsymbol{\varepsilon}}(\boldsymbol{\rho}) \right|.
\end{align*}
Since
$w(\rho) \ll \delta,$ this is
\begin{align*}
\ll_{A,k} \left( T^{-c_k}U \right)^{-A}\delta^k\sum_{\boldsymbol{\varepsilon} \in \{ \pm 1\}^k }\sum_{\boldsymbol{\gamma}\in [0,T]^k} 1
 &\ll_{A,k}  \left( T^{-c_k}U \right)^{-A} \delta^k  N(T)^k\\
&\ll_{A,k, \varepsilon}  U^{\frac{k}{c_k}-\varepsilon A} \delta^k 
\end{align*}
by the condition that $T \leq U^{(1-\varepsilon)/c_k}.$ Since it is assumed that $c_k>k,$ by taking $A=100/\varepsilon,$ this is $\ll_{k,\varepsilon} U^{-99}\delta^k.$ As the conditions imply that $1/\delta \leq T \leq U^{1/c_k} \leq U^{1/k},$ we have
\begin{align} \label{eq:sum3}
\Sigma_3 \ll_k \delta^{100k}.
\end{align}

On the other hand, assuming Conjecture \ref{conj: qli} in the form of (\ref{eq:g(T)}), the sum $\Sigma_2$ is
\begin{align*}
\ll_k \delta^k
\underset{0<\left|\langle \boldsymbol{\varepsilon}, \boldsymbol{\gamma} \rangle \right| \leq T^{-c_k}}{\sum \sum} 1
\ll_k \delta^k N(T)^{k/2} /g(T).
\end{align*} 
Since $T \leq g(T)^{1/k} \delta^{-1},$ this is
\begin{align} \label{eq:sum2}
\ll_k \left( \delta \log \frac{1}{\delta} \right)^{k/2}g(T)^{-\frac{1}{2}}
=o_k \left( \left( \delta \log \frac{1}{\delta} \right)^{k/2}\right).
\end{align}

We are left with the sum $\Sigma_1.$ As discussed above, assuming LI, the alternating sum  $\langle \boldsymbol{\varepsilon}, \boldsymbol{\gamma} \rangle \Mod{1}$
vanishes if and only if $k$ is even and $[k]$ can be partitioned into $k/2$ pairs $\{j,j'\} $
of which $\varepsilon_{j'}=-\varepsilon_{j}$ and $\gamma_{j'}=\gamma_j.$ If $k$ is odd, then the sum $\Sigma_1$ is empty. Otherwise, since $W$ is of unit mass, i.e., $\widehat{W}(0)=1,$ we have
\begin{align*}
\Sigma_1=\underset{\substack{\langle \boldsymbol{\varepsilon}, \boldsymbol{\gamma} \rangle=0\\|\{ \gamma_1,\ldots,\gamma_k \}|=k/2}}{\sum \sum}
w^{\boldsymbol{\varepsilon}}(\boldsymbol{\rho})+\underset{\substack{\langle \boldsymbol{\varepsilon}, \boldsymbol{\gamma} \rangle=0\\|\{ \gamma_1,\ldots,\gamma_k \}|<k/2}}{\sum \sum}w^{\boldsymbol{\varepsilon}}(\boldsymbol{\rho})
=:\Sigma_{1,1}+\Sigma_{1,2}.
\end{align*}
Let us deal with the sum $\Sigma_{1,2}$ first.  In this case, note that there are at least four repeated $\gamma$'s, say for instance $\gamma_1=\gamma_2=\gamma_3=\gamma_4.$ Since $w(\rho)\ll \delta,$ pulling $w(\gamma_1)\overline{w(\gamma_2)}$ out of the sum gives
\begin{align} \label{eq:sum12}
\Sigma_{1,2} &\ll_k \delta^2 \left( \sum_{0<\gamma \leq T} |w(\rho)|^2 \right)^{\frac{k}{2}-1}
\nonumber \\
&\ll_k \frac{\delta}{\log \frac{1}{\delta}} \cdot \left(\delta \log \frac{1}{\delta} \right)^{\frac{k}{2}}.
\end{align}
Meanwhile, since there are $2^{k/2}$ ways of choosing $\boldsymbol{\varepsilon}$ and $\mu_k$ ways of partitioning $k$ distinct $\gamma's$ into $k/2$ unordered pairs, the sum $\Sigma_{1,1}$ is
\begin{align*}
2^{k/2}\mu_k \underset{\substack{0<\gamma_1,\ldots, \gamma_{k/2} \leq T\\ \text{distinct}}}{\sum \cdots \sum}
 \prod_{j=1}^{k/2} |w(\rho_j)|^2=
 \mu_k \underset{\substack{|\gamma_1|,\ldots, |\gamma_{k/2}| \leq T\\ \text{distinct}}}{\sum \cdots \sum}
 \prod_{j=1}^{k/2} |w(\rho_j)|^2.
\end{align*}
Similar to (\ref{eq:sum12}), we have
\begin{align*}
\underset{\substack{|\gamma_1|,\ldots, |\gamma_{k/2}| \leq T\\ \text{distinct}}}{\sum \cdots \sum}
 \prod_{j=1}^{k/2} |w(\rho_j)|^2
=\left( \sum_{|\gamma| \leq T} |w(\rho)|^2   \right)^{k/2}
+O_k \left( \frac{\delta}{\log \frac{1}{\delta}} \cdot \left(\delta \log \frac{1}{\delta} \right)^{k/2} \right),
\end{align*}
and hence by definition
\begin{align} \label{eq:sum11}
\Sigma_{1,1}=\mu_k \cdot ( V^{(T)}(\delta,0))^{k/2}+O_k 
 \left( \frac{\delta}{\log \frac{1}{\delta}} \cdot \left(\delta \log \frac{1}{\delta} \right)^{k/2} \right).
\end{align}
As $\mu_k=0$ when $k$ is odd, this holds for any integer $k \geq 1.$ Since 
\begin{align*}
V^{(T)}(\delta;0)=\delta \log \frac{1}{\delta} +O\left(\delta+ \frac{\log T}{T} \right)
\end{align*}
by Proposition \ref{thm:cov} and Lemma \ref{lem:VT}, the proposition follows from
combining (\ref{eq:sum3}), (\ref{eq:sum2}), (\ref{eq:sum12}) and (\ref{eq:sum11}).
\end{proof}

\begin{lemma} \label{lem:chebyshev2}
Let $T, U \geq 2$ with $T \leq e^U.$ Suppose $\delta \to 0^{+}$ as $T \to \infty.$ Then 
\begin{align*}
\mathbb{E}_{u \sim U}^W
( ( E\left( e^u;\delta,0 \right)-E^{(T)}\left( e^u;\delta,0 \right))^2 )\ll \frac{\log^2 T}{TU}+U^3e^{-U/2}.
\end{align*}

\begin{proof}
This is similar to Lemma \ref{lem:chebyshev}, except we are integrating with respect to a smooth weight $W$ for extra savings. Let $u \in [U,2U].$  Invoking
Lemma \ref{lem:explicit} gives
\begin{align*}
E\left( e^u;\delta,0 \right)-E^{(T)}\left( e^u;\delta,0 \right)
=-\sum_{T<|\gamma| \leq e^U}w(\rho)e^{i\gamma u}
+O ( U^2 e^{-u/2} ).
\end{align*}
On one hand, since $W$ is supported on $(1/2,5/2),$ we have
\begin{align*}
\frac{1}{U} \int_{-\infty}^{\infty} 
\left| U^2 e^{-u/2} \right|^2 W\left( \frac{u}{U} \right) du \ll U^3e^{-U/2}.
\end{align*}
On the other, opening the square gives
\begin{align*}
\frac{1}{U} \int_{-\infty}^{\infty} \left| -\sum_{T<|\gamma|\leq e^U} w(\rho)e^{i\gamma u} \right|^2 W\left( \frac{u}{U} \right) du
=\sum_{T<|\gamma_1|, |\gamma_2| \leq e^U} w(\rho_1)\overline{w(\rho_2)} \cdot
\widehat{W}\left( -U(\gamma_1-\gamma_2) \right).
\end{align*}
Since $w(\rho) \ll |\gamma|^{-1}$ and $\widehat{W}(x) \ll (1+|x|)^{-2},$ this is
\begin{align*}
 \ll \sum_{\substack{|\gamma_1|, |\gamma_2|>T\\|\gamma_1-\gamma_2| \leq U^{-1}}} \frac{1}{|\gamma_1\gamma_2|} 
+ \frac{1}{U^2}\sum_{\substack{|\gamma_1|, |\gamma_2| > T\\|\gamma_1-\gamma_2| > U^{-1}}} \frac{1}{|\gamma_1\gamma_2||\gamma_1-\gamma_2|^2}.
\end{align*}
By partial summation, the first sum is
\begin{align} \label{eq:finalsum}
\ll \underset{\substack{x,y > T\\|x-y|\leq U^{-1}}}{\iint} \frac{\log x \log y}{xy} \cdot dxdy 
\ll& \int_{x>T} \frac{\log^2 x}{x^2} \left(\int_{\substack{y>T\\ |x-y| \leq U^{-1}}} dy \right) dx \nonumber\\
\ll& \frac{\log^2 T}{TU},
\end{align}
and the second sum is
\begin{align*}
\ll \frac{1}{U^2}\underset{\substack{x,y > T\\|x-y| > U^{-1}}}{\iint}
\frac{\log x \log y}{xy} \cdot \frac{ dxdy }{(x-y)^2}.
\end{align*}
We split the integral into
\begin{gather}
I_1:= 
\frac{1}{U^2}  \underset{\substack{x,y > T\\U^{-1}<|x-y| \leq U^{-1}(\log T/\log^2 T)}}{\iint} 
\frac{\log x \log y}{xy} \cdot \frac{ dxdy }{(x-y)^2} \nonumber \\
\leq \frac{1}{U^2} \sum_{1 \leq j \leq T/\log^2 T}  
 \underset{\substack{x,y > T\\jU^{-1}<|x-y| \leq (j+1)U^{-1}}}{\iint} 
\frac{\log x \log y}{xy} \cdot \frac{ dxdy }{(x-y)^2} \nonumber \\
\ll \frac{\log^2 T}{TU} \label{eq:finali1}
\end{gather}
and
\begin{align}
I_2:=\frac{1}{U^2}  \underset{\substack{x,y > T\\|x-y|>U^{-1}(T/\log^2 T)}}{\iint} 
\frac{\log x \log y}{xy} \cdot \frac{ dxdy }{(x-y)^2} \ll \frac{\log^2 T}{TU} \label{eq:finali2}.
\end{align}
Combining (\ref{eq:finalsum}), (\ref{eq:finali1}) and (\ref{eq:finali2}), the lemma follows.
\end{proof}
\end{lemma}

\begin{proof}[Proof of Theorem \ref{thm:short}]
Given a Borel subset $B \subseteq \mathbb{R},$ we denote
\begin{align*}
\mathbb{P}_{u \sim U}^W \left( B \right):=\frac{1}{U}\int_B W\left( \frac{u}{U} \right) du.
\end{align*}
Let $T \geq 2$ be a real number satisfying $\delta T \to \infty$ but $\log 
 \left( \delta T \right)=o \left( \log g(T) \right),$ and $\log T =o \left( \log U\right)$ as $U \to \infty,$ so that Proposition \ref{prop:moments} is applicable for any integer $k \geq 1.$
Applying the method of moments (see \cite[Theorem 30.2]{MR1324786}), the proposition implies that
\begin{align*}
\mathbb{P}_{u \sim U}^W ( \widetilde{E}^{(T)}\left( e^u;\delta,0 \right) \in (\alpha, \beta] )=\frac{1}{\sqrt{2\pi}}\int_{\alpha}^{\beta}e^{-\frac{t^2}{2}}dt+o(1)
\end{align*}
for any fixed real numbers $\alpha<\beta.$ Applying Proposition \ref{thm:cov} and Lemma \ref{lem:VT}, we have
\begin{align*}
V^{(T)}(\delta,0)=(1+o(1))V(\delta,0), 
\end{align*}
so that
\begin{align*}
\mathbb{P}_{u \sim U}^W \left( \frac{E^{(T)}\left( e^u;\delta,0 \right)}{\sqrt{V(\delta,0)}} \in (\alpha, \beta] \right)=\frac{1}{\sqrt{2\pi}}\int_{\alpha}^{\beta}e^{-\frac{t^2}{2}}dt+o(1).
\end{align*}
Since 
\begin{align*}
\mathbb{E}_{u \sim U}^W
( ( E\left( e^u;\delta,0 \right)-E^{(T)}\left( e^u;\delta,0 \right))^2 ) =o\left( V(\delta,0) \right)
\end{align*}
by Proposition \ref{thm:cov} and Lemma \ref{lem:chebyshev2}, it follows from Chebyshev's inequality that
\begin{align*}
\mathbb{P}_{u \sim U}^W ( \widetilde{E}\left( e^u;\delta,0 \right) \in (\alpha, \beta] )=\frac{1}{\sqrt{2\pi}}\int_{\alpha}^{\beta}e^{-\frac{t^2}{2}}dt+o(1).
\end{align*}
Finally, given $\varepsilon \in (0,1),$ We choose $W$ such that $W \equiv 1$ on $(1+\varepsilon, 2-\varepsilon).$ Since $W$ is nonnegative and of unit mass, we have
\begin{gather*}
\frac{1}{U}\mathop{\mathrm{meas}}\left\{ x \in [U,2U] \,: \, \widetilde{E}\left( e^u;\delta,0 \right) \in (\alpha,\beta] \right\}-\mathbb{P}_{u \sim U}^W ( \widetilde{E}\left( e^u;\delta,0 \right) \in (\alpha, \beta] ) \\
=\frac{1}{U}\int_{-\infty}^{\infty} 1_{(\alpha,\beta]} 
( \widetilde{E}\left( e^u;\delta,0 \right) )
\left( 1_{[1,2]}\left( \frac{u}{U} \right)-W\left( \frac{u}{U} \right) \right) du \\
\leq  \int_{\mathbb{R} \setminus (1+\varepsilon, 2-\varepsilon)} 
\left(1_{[1,2]}(x)+W(x) \right) dx = 4\varepsilon.
\end{gather*}
Letting $\varepsilon \to 0^+,$ the proof is completed.
\end{proof}


Assuming QLI, it is plausible to compute the mixed moments of $\widetilde{\boldsymbol{E}}\left(e^u;\delta,\boldsymbol{t}\right)$ for $r \geq 2$ and establish a multidimensional analog of Theorem \ref{thm:short}. However, we refrain from pursuing this here, as our primary objective lies in examining the uniformity of $\delta$ in $U.$

\section{Open questions}

As discussed in the introduction, Montgomery and Soundararajan \cite{MR2104891} established a central limit theorem conditionally for primes in a very short moving interval, provided that $\frac{H}{\log N} \to \infty$ and $\frac{\log H}{\log N} \to 0$ as $N \to \infty.$ Our Theorem \ref{thm:short}, on the other hand, establishes a central limit theorem conditionally, provided that $h=h(x)=\delta x$ with $\delta=\delta(X)>(\log X)^{-\varepsilon}$ for any $\varepsilon>0$ in logarithmic scale, however. Given the state of affairs, one may ask how primes in a short moving interval behave in the intermediate range. In particular, how does the transition from natural to logarithmic density occur?

\section*{Acknowledgements}
The author is grateful to Andrew Granville and Youness Lamzouri for their advice and encouragement. He would also like to thank R\'{e}gis de la Bret\`eche for insightful discussions, Kannan Soundararajan for his valuable suggestions, and Cihan Sabuncu for carefully reading an earlier version of the manuscript. In particular, he is indebted to the anonymous referees for their thoughtful comments and corrections. 

The latter part of this work was supported by the Swedish Research Council under grant no. 2016-06596 while the author was in residence at Institut Mittag-Leffler in Djursholm, Sweden during the semester of Winter 2024.

\printbibliography


\end{document}